\documentclass[11pt,a4paper]{amsart}
\usepackage[left=3.5cm,right=3.5cm,top=3.5cm,bottom=3.5cm]{geometry}
\usepackage{enumerate}
\usepackage{leftidx,color}
\usepackage{mathtools}
\usepackage{amsmath,amscd,amssymb,amsthm}
\usepackage[arrow,matrix]{xy}
\usepackage[OT2,T1]{fontenc}
\usepackage{graphicx}
\usepackage{todonotes}
\usepackage{thmtools,thm-restate}
\usepackage[colorlinks,linkcolor=blue,anchorcolor=blue,citecolor=blue,backref=page]{hyperref}

\DeclareMathOperator{\gal}{Gal}

\DeclareMathOperator{\aut}{Aut}
\DeclareMathOperator{\ab}{ab}
\DeclareMathOperator{\GL}{GL}
\DeclareMathOperator{\PGL}{PGL}
\DeclareMathOperator{\en}{End}

\theoremstyle{definition}
\newtheorem{definition}{Definition}[section]
\newtheorem{example}[definition]{Example}
\newtheorem{remark}[definition]{Remark}
\newtheorem*{remark*}{Remark}

\theoremstyle{plain}
\newtheorem{theorem}[definition]{Theorem}
\newtheorem*{theorem*}{Theorem}
\newtheorem{corollary}[definition]{Corollary}
\newtheorem{lemma}[definition]{Lemma}

\newtheorem{conjecture}[definition]{Conjecture}
\newtheorem{theoremA}{Theorem}
\newtheorem{theoremAA}{Theorem}

\newcommand{\N}{{ \mathbb N }}
\newcommand{\Q}{{ \mathbb Q }}
\newcommand{\Z}{{ \mathbb Z }}
\newcommand{\R}{{ \mathbb R }}
\newcommand{\PP}{{ \mathbb P }}
\def \C {{\mathbb C}}
\renewcommand{\ab}{{ \text{ab} }}
\newcommand{\p}{{ \mathfrak p }}
\renewcommand{\P}{{ \mathfrak P }}

\renewcommand{\O}{{ \mathcal O }}
\def\ov\K{\overline{K}}

\makeatletter
\@namedef{subjclassname@2020}{\textup{2020} Mathematics Subject Classification}
\makeatother

\usepackage[norefs,nocites]{refcheck}

\author[A. Ferraguti]{Andrea Ferraguti}
\address{Scuola Normale Superiore, Piazza dei Cavalieri 7, 56126 Pisa}
\address{DICATAM, Università degli Studi di Brescia, via Branze 43, I-25123 Brescia}
\email{and.ferraguti@gmail.com}

\author[A. Ostafe]{Alina Ostafe}
\address{School of Mathematics and Statistics, University of New South Wales, Sydney, NSW 2052, Australia}
\email{alina.ostafe@unsw.edu.au}

\author[U. Zannier]{Umberto Zannier}
\address{Scuola Normale Superiore, Piazza dei Cavalieri 7, 56126 Pisa}
\email{umberto.zannier@sns.it}

\title[Cyclotomic and abelian points in backward orbits]{Cyclotomic and abelian points in backward orbits of rational functions}

\begin{document}

\begin{abstract}
We prove several results on backward orbits of rational functions over number fields.  First, we show that if $K$ is a number field, $\phi\in K(x)$ and $\alpha\in K$ then the extension of $K$ generated by the abelian points (i.e.\ points that generate an abelian extension of $K$) in the backward orbit of $\alpha$ is ramified only at finitely many primes. This has the immediate strong consequence that if all points in the backward orbit of $\alpha$ are abelian then $\phi$ is post-critically finite. We use this result to prove two facts: on the one hand,  if $\phi\in \Q(x)$ is a quadratic rational function not conjugate over $\Q^{\ab}$ to a power or a Chebyshev map and all preimages of $\alpha$ are abelian, we show that $\phi$ is $\Q$-conjugate to one of two specific quadratic functions, in the spirit of a recent conjecture of Andrews and Petsche. On the other hand we provide conditions on a quadratic rational function in $K(x)$ for the backward orbit of a point $\alpha$ to only contain finitely many cyclotomic preimages, extending previous results of the second author.  Finally, we give necessary and sufficient conditions for a triple $(\phi,K,\alpha)$, where $\phi$ is a $K$-Latt\`es map over a number field $K$ and $\alpha\in K$, for the whole backward orbit of $\alpha$ to only contain abelian points.
\end{abstract}

\maketitle

\section{Introduction}

This paper is motivated by the following question. Suppose that $K$ is a number field with algebraic closure $\overline{K}$, $\phi\in K(x)$ is a rational function of degree  greater than $1$ and $\alpha\in K$. Let $\phi^{-\infty}(\alpha)$ be the backward orbit of $\alpha$, namely the set of all $\gamma\in \overline{K}$ such that $\phi^n(\gamma)=\alpha$ for some $n\geq 1$, where $\phi^n$ is the $n$-th iterate of $\phi$. How large is the set of all $\gamma \in \phi^{-\infty}(\alpha)$ that belong to the \emph{cyclotomic closure} $K^c$ of $K$? And how large is the set of all $\gamma \in \phi^{-\infty}(\alpha)$ that belong to the \emph{abelian closure} $K^{\ab}$ of $K$?

We say that two rational functions $\phi,\psi$ are \emph{$L$-conjugate} if there is an extension $L/K$ and some invertible $f=\frac{ax+b}{cx+d}\in L(x)$ such that $f\circ \phi\circ f^{-1}=\psi$.

In their paper~\cite{DZ}, Dvornicich and the third author introduced a technique, based on ideas rooting back to Mordell, Lang, and Loxton, that uses the Torsion Coset Theorem~\cite[Theorem~4.2.2]{BG}, to deal with a similar problem. In particular, they proved that if $\phi$ is a polynomial of degree $d\ge 2$, then $\phi$ has only finitely many preperiodic points in $K^c$, unless $\phi$ is $\overline{K}$-conjugate to a polynomial of a very special kind, namely $x^d$ or $\pm T_d(x)$, where $T_d$ is the Chebyshev polynomial of degree $d$.  This suggests the idea that the existence of infinitely many cyclotomic points in a backward orbit should be an extremely rare phenomenon. The second author, in~\cite{ostafe}, used the technique of Dvornicich and Zannier~\cite[Theorem~2]{DZ} to prove a similar finiteness result for the presence of cyclotomic points in backward orbits of roots of unity. In particular, as a consequence of~\cite[Theorem~1.2]{ostafe}, if $\phi=f/g\in K(x)$, where $f,g\in K[x]$ are such that $\deg f>\deg g +1$, and $\alpha$ is a root of unity, then $\phi^{-\infty}(\alpha)$ contains only finitely many cyclotomic points  unless $\phi$ is $\overline{K}$-conjugate to $x^d$ or $\pm T_d(x)$. The condition on the degrees means, in dynamical terms, that $\infty$ is a superattracting fixed point for $\phi$. Later, in~\cite[Theorem~1.5]{Chen} Chen generalized this result to backward orbits of regions of bounded house, but in its essence the technique does not change and the condition on the degrees is still needed.

On the other hand, the arithmetic of backward orbits of rational functions has attracted a lot of attention in recent years. Motivated by the seminal work of Odoni~\cite{odoni}, Boston and Jones~\cite{boston2} started studying the Galois structure of the extension of $K$ generated by the backward orbit of a point, a topic that intriguingly recalls the theory of Tate modules of elliptic curves. Several authors studied the problem in detail (see, for example,~\cite{anderson,benedetto3,bridy2,ferra4,hindes,jones_divisors,juul,lukina}, or the survey~\cite{benedetto4}), but only in 2020 Andrews and Petsche~\cite{andrews} formulated the conjecture below, aiming to characterize all pairs $(\phi,\alpha)$, where $\phi\in K[x]$ and $\alpha\in K$, such that $\phi^{-\infty}(\alpha)$ generates an abelian extension of $K$. 

We say that a pair $(\phi,\alpha)$ is $L$-conjugate to a pair $(\psi,\beta)$ if there exists an extension $L/K$ and an invertible $f=\frac{ax+b}{cx+d}\in L(x)$ such that $f\circ\phi\circ f^{-1}=\psi$ and $f(\alpha)=\beta$. Moreover, $\alpha$ is \emph{exceptional} for $\phi$ if $\phi^{-\infty}(\alpha)$ is finite.

\begin{conjecture}[{{\cite[Conjecture 1]{andrews}}}]\label{conjecture}
Let $K$ be a number field, let $\phi\in K[x]$ have degree $d\ge 2$ and let $\alpha\in K$. Suppose that $\alpha$ is not exceptional for $\phi$. Then $\phi^{-\infty}(\alpha)\subseteq K^{\ab}$ if and only if the pair $(\phi,\alpha)$ is $K^{\ab}$-conjugate to either $(x^d,\zeta)$ or $(\pm T_d,\zeta+\zeta^{-1})$ for some root of unity $\zeta$.\footnote{The formulation in~\cite{andrews} overlooks the case of polynomials of the form $-T_d(x)$, that needs to be included as well, since when $d$ is odd $T_d(x)$ and $-T_d(x)$ are not $\overline{K}$-conjugate.}
\end{conjecture}

Conjecture~\ref{conjecture} has been proven by Andrews and Petsche~\cite[Theorem 1]{andrews} for $K=\Q$, $\deg \phi=2$ and $\alpha$ such that $(\phi,\alpha)$ is \emph{stable}, i.e.\ such that $\phi^n(x)-\alpha$ is irreducible for every $n$, by Pagano and the first author~\cite{ferra1} for $K=\Q$, $\deg \phi=2$ and all $\alpha$, and of course it follows from the second author's result~\cite{ostafe} for $K=\Q$ and any $\phi \in \Q[x]$ and $\alpha\in \Q$ (see Corollary~\ref{ap_conj_polys}). Interestingly enough, the three papers use three completely different sets of techniques (see the first author's survey \cite{ferra5} for an explanation of the strategies). Nevertheless, both the arguments in~\cite{andrews} and~\cite{ferra1} show, as a first step, that if $\phi$ is a quadratic polynomial and $\phi^{-\infty}(\alpha)$ only contains abelian points, then $\phi$ is \emph{post-critically finite} (PCF), i.e.\ the critical points of $\phi$ all have finite orbits. This is already a rather strong restriction on $\phi$, since as shown in~\cite{benedetto2} post-critically finite maps that are not flexible Latt\`es maps are a set of bounded height in the moduli space of rational functions of a given degree.

Our first result, Theorem~\ref{finite_ramification}, widely generalizes the aforementioned implication, in the following way.

\begin{theoremAA} 
Let $K$ be a number field, let $\phi\in K(x)$ be a rational function of degree $\ge 2$ and let $\alpha\in K$. Let $K_\infty^{\ab}(\phi,\alpha)$ be the extension of $K$ generated by all $\gamma\in \phi^{-\infty}(\alpha)$ that belong to an abelian extension of $K$. Then $K_\infty^{\ab}(\phi,\alpha)/K$ ramifies only at finitely many primes. Consequently, if $\phi^{-\infty}(\alpha)\subseteq K^{\ab}$, then $\phi$ is PCF.
 \end{theoremAA}

The last part of the statement is an immediate consequence of~\cite[Theorem~5]{bridy}, which shows that if the set of primes of $K$ that ramify in the extension of $K$ generated by the backward orbit of $\alpha$,  $K_\infty(\phi,\alpha)$, where $\alpha$ is not exceptional, is a finite set then $\phi$ is PCF.

It is actually reasonable to expect that $K_\infty^{\ab}(\phi,\alpha)/K$ is a finite extension except for very special cases. In fact, in Corollary~\ref{abelian_implies_pcf} we show that if~\cite[Conjecture~1.2]{jones_ev} holds true, then $[K_\infty^{\ab}(\phi,\alpha):K]=\infty$ implies that $\phi$ is PCF. Theorems~\ref{ap_conjecture} and~\ref{suff_conditions} below prove that for several large classes of $\phi\in \Q(x)$ it is indeed the case that $[\Q_\infty^{\ab}(\phi,\alpha):\Q]<\infty$.

Theorem~\ref{finite_ramification} is strong evidence towards Conjecture~\ref{conjecture}, since power and Chebyshev maps are PCF, and moreover over any fixed number field $K$ there exist only finitely many $\overline{K}$-conjugacy classes of PCF polynomials.

Of course one can wonder what shape should Andrews-Petsche's conjecture have for rational functions. Clearly the statement of the conjecture cannot hold true verbatim for all rational functions (even including the obvious extra case where $\phi$ is $\overline{K}$-conjugate to $x^{-d}$), because of the existence of Latt\`es maps (see Section~\ref{lattes}). However, when $\phi \in \Q(x)$ is a quadratic map the theory of complex multiplication shows that $\phi$ cannot be a $\Q$-Latt\`es map (cf.\ Definition \ref{lattes_def}). In Theorem~\ref{ap_conjecture}, we use Theorem~\ref{finite_ramification} to prove the result below, in the spirit of Conjecture~\ref{conjecture}. In order to state Theorem~\ref{ap_conjecture}, we introduce a preliminary definition. Recall that for $K$ a number field, $K^c$ denotes the cyclotomic closure of $K$.

\begin{definition}
A degree $d$ rational function $\phi\in K(x)$ is called \emph{special} if it is $K^c$-conjugate to $x^{\pm d}$ or $\pm T_d(x)$, and \emph{non-special} otherwise.
\end{definition}

We refer to Lemma \ref{conjugacy_field} to explain why conjugacy is required to happen over $K^c$ rather than over $\overline{K}$.

\begin{theoremAA}
Let $\phi\in \Q(x)$ be a non-special rational map of degree $d\geq 2$, let $\alpha\in \Q$ be a non-exceptional point for $\phi$ and suppose that $G_\infty(\phi,\alpha)$ is abelian. Then the following hold.
\begin{enumerate}
\item $\phi$ is PCF and all of its critical points are strictly preperiodic.
\item If $d=2$, then $\phi$ is $\Q$-conjugate to either $\frac{2x^2+4x+4}{-x^2+4}$ or $\frac{1}{2(x-1)^2}$.
\end{enumerate}\end{theoremAA}

Theorem~\ref{ap_conjecture} is proven via an extension of the results of~\cite{ostafe} (see Lemma~\ref{superattracting_point}), combined with the classification given in~\cite{lukas} of PCF quadratic maps over the rationals.

Part $(1)$ of Theorem \ref{ap_conjecture} follows from Theorem \ref{finite_ramification} together with Lemma~\ref{superattracting_point} below, that states that if $\phi\in K(x)$ has a periodic critical point (or, equivalently, if some iterate of $\phi$ has a superattracting fixed point) then $\phi^{-\infty}(\alpha)$ contains only finitely many cyclotomic points unless $\phi$ is special. Note that $\phi^{-\infty}(\alpha)$ containing only finitely many cyclotomic points is equivalent to $[K^c_\infty(\phi,\alpha):K]<\infty$, where $K^c_\infty(\phi,\alpha)$ is defined as in Theorem~\ref{suff_conditions}, see Remark~\ref{abelian_part_properties}. 

Of course one can wonder what happens when the condition on superattractivity, which is crucial in the proof of Lemma~\ref{superattracting_point}, is removed. In Theorem~\ref{suff_conditions} we focus on quadratic rational functions over an arbitrary number field and we prove the same conclusion of Lemma~\ref{superattracting_point} for a different class of rational functions, as follows, using a completely new argument. Notice that these functions need not having a periodic critical point, so they can fall short of the hypotheses of Lemma~\ref{superattracting_point}. We denote by $h(\cdot)$ the absolute logarithmic Weil height on $\overline{\Q}$.

\begin{theoremAA}
Let $K$ be a number field, let $a,c\in K$ with $a\neq 0$ and let $\phi\coloneqq ax+a/x+c$. Let $\alpha\in K$, and suppose that there exists a prime $\p$ of $K$ lying above $2$ such that one of the following conditions hold:
\begin{enumerate}
\item $v_\p(a)=0$ and $v_\p(c)>0$;
\item $v_\p(a)=0$, $h(a)>\log 4$ and $c\equiv 1\bmod \p$.
\end{enumerate}
Let $K_\infty^c(\phi,\alpha)$ be the extension of $K$ generated by all $\gamma\in \phi^{-\infty}(\alpha)$ that belong to a cyclotomic extension of $K$. Then $[K^c_\infty(\phi,\alpha):K]<\infty$.
\end{theoremAA}

The proof makes an essential use of Theorem~\ref{finite_ramification}; point (2) needs the results of~\cite{benedetto2} and~\cite{jones_ev}. The hypotheses of Theorem~\ref{suff_conditions} are very different from those of~\cite{Chen,DZ,ostafe}: maps that satisfy $(1)$ or $(2)$ need not have a periodic critical point (in fact, the forward orbits of the critical points can be both infinite). From the dynamical point of view, the existence of a prime $\p$ with $v_\p(a)=0$ implies that the map has good reduction at $\p$, and therefore its $\p$-adic Julia set is necessarily empty (see \cite[Proposition 5.11]{benedetto1}). Notice that writing $\phi$ as $ax+a/x+c$ is not restrictive with regard to the conclusion we want to obtain, since every quadratic rational function is conjugate to one in that shape over a finite extension of $K$.

%
%

Finally, in Section~\ref{lattes} we study abelian points in backward orbits of Latt\`es maps. We will use a slightly stronger notion of Latt\`es map, that keeps track of the fields of definition of the objects involved.

\begin{definition}\label{lattes_def}
Let $K$ be a number field. A \emph{$K$-Latt\`es map} is a rational map $\phi\in K(x)$ that fits in a commutative diagram of the following form:
$$
\xymatrix{
    E \ar[r]^{\psi} \ar[d]_{\pi} & E \ar[d]^{\pi} \\
    \mathbb P^1 \ar[r]^{\phi}       & \mathbb P^1},
$$
where $E$ is an elliptic curve defined over $K$ and $\pi,\psi$ are non-constant morphisms defined over $K$.
\end{definition}

Notice that, as an element of $K(x)$, a $K$-Latt\`es map can be defined over a smaller field than $K$, cf.\ Example \ref{lattes_example}.

If $\phi$ is a $K$-Latt\`es map and $\alpha\in K$, complex multiplication theory easily gives sufficient conditions for $\phi^{-\infty}(\alpha)$ to generate an abelian extension of $\alpha$; we prove in Theorem~\ref{lattes_abelian} that these are also necessary, as follows.

We recall that $K^{\ab}$ is the abelian closure of $K$.

\begin{theoremAA}
Let $K$ be a number field, let $\phi\in K(x)$ be a $K$-Latt\`es map of degree $d\geq 2$ associated to an elliptic curve $E$ and let $\alpha\in K$. Then $\phi^{-\infty}(\alpha)\subseteq K^{\ab}$ if and only if the following three conditions hold:
\begin{enumerate}
\item $\alpha$ is preperiodic for $\phi$;
\item $E$ has CM by an order in a quadratic number field $F$;
\item the field $F=\text{End}(E)\otimes \Q$ is contained in $K$.
\end{enumerate}
\end{theoremAA}

The proof makes use of Silverman's result~\cite{silverman3} on the canonical height of abelian points and of Lozano-Robledo's classification of images of Galois representations of CM elliptic curves~\cite{lozano}.

\subsection*{Notation}
Whenever $K$ is a number field, we implicitly fix an algebraic closure $\overline{K}$. The cyclotomic closure of $K$ inside $\overline{K}$ is denoted by $K^c$, while $K^{\ab}$ denotes the abelian closure. We denote by $G_K$ the absolute Galois group of $K$.

If $\phi,\psi\in K(x)$ and $L/K$ is an extension of fields, we say that $\phi,\psi$ are \emph{$L$-conjugate} if there exists $f=\frac{ax+b}{cx+d}\in L(x)$ with $ad-bc\neq 0$, such that $\psi= f^{-1}\circ\phi\circ f$. If moreover $\alpha,\beta\in K$ we say that the pairs $(\phi,\alpha)$ and $(\psi,\beta)$ are $L$-conjugate if there exists an $f$ as above such that $\psi= f^{-1}\circ\phi\circ f$ and $f^{-1}(\alpha)=\beta$.

If $\phi\in K(x)$ and $n\ge 0$, we denote by $\phi^n$ the $n$-th iterate of $\phi$, setting by convention $\phi^0=x$. If $\alpha\in K$, we let

$$\phi^{-n}(\alpha)\coloneqq\{\gamma\in\overline{K}\colon \phi^n(\gamma)=\alpha\} \mbox{ and } \phi^{-\infty}(\alpha)\coloneqq\bigcup_{n\geq 0}\phi^{-n}(\alpha).$$

The latter is called the \emph{backward orbit} of $\alpha$. If $\phi^n=p_n/q_n$ for some coprime $p_n,q_n\in K[x]$ and $\alpha\in K$ then we let $K_n(\phi,\alpha)$ be the splitting field of $p_n-\alpha q_n$ and $G_n(\phi,\alpha)$ its Galois group. Finally, $K_\infty(\phi,\alpha)\coloneqq \varinjlim_n K_n(\phi,\alpha)$ and $G_\infty(\phi,\alpha)\coloneqq \varprojlim_n G_n(\phi,\alpha)=\gal(K_\infty(\phi,\alpha)/K)$. Equivalently, $K_\infty(\phi,\alpha)$ is the extension of $K$ generated by the backward orbit of $\alpha$.

For a rational function $\phi\in K(x)$, the \emph{critical points} of $\phi$ are the ramification points of $\phi$, regarded as a self-morphism of the projective line. Also, we say that an element $\alpha\in\ov\K$ is \emph{exceptional} for $\phi$ if its backward orbit is finite.

We denote by $h(\cdot)$ the absolute logarithmic Weil height.

\section{Abelian points and finite ramification}
Let $K$ be a number field, $\phi\in K(x)$ and $\alpha\in K$. In this section we show that points in the backward orbit of $\alpha$ that are abelian over $K$ generate altogether an extension of $K$ that is only ramified at finitely many primes. It is reasonable to expect that this extension must even be finite, unless $\phi$ is of a very special form. We will turn our attention to this issue in Section~\ref{cyclo_quadratic}.

\begin{definition}
We say that $\beta\in \phi^{-\infty}(\alpha)$ is an \emph{abelian point} if the extension $K(\beta)/K$ is abelian. We call \emph{abelian part} of $K_\infty(\phi,\alpha)$, and we denote it by $K_\infty^{\ab}(\phi,\alpha)$, the extension of $K$ generated by all abelian points of $\phi^{-\infty}(\alpha)$. 
\end{definition}

\begin{remark}\label{abelian_part_properties}
Notice that $[K_\infty^{\ab}(\phi,\alpha):K]$ is infinite if and only if $\phi^{-\infty}(\alpha)$ contains infinitely many distinct abelian points. In fact, one direction is obvious, while the other one follows from the Northcott property for number fields, together with the fact that points in $\phi^{-\infty}(\alpha)$ have bounded height. Same remark applies if we replace abelian points with cyclotomic points.

The abelian part defined above is \emph{not} the largest abelian extension of $K$ contained in $K_\infty(\phi,\alpha)$. For example if $K=\Q$, $\phi=x^2$ and $\alpha=2$, the abelian part of $\Q_\infty(\phi,\alpha)$ is $\Q(\sqrt{2})$, but $\Q_\infty(\phi,\alpha)$ contains $\Q(\zeta_{2^{\infty}})$.
\end{remark}

\begin{theoremA}\label{finite_ramification}
Let $K$ be a number field, let $\phi\in K(x)$ be a rational function of degree $\ge 2$ and let $\alpha\in K$. Let $K_\infty^{\ab}(\phi,\alpha)$ be the extension of $K$ generated by all $\gamma\in \phi^{-\infty}(\alpha)$ that belong to an abelian extension of $K$. Then $K_\infty^{\ab}(\phi,\alpha)/K$ ramifies only at finitely many primes.
\end{theoremA}

\begin{proof}
In order to prove the first claim, start by noticing that it suffices to prove it for a finite extension of $K$. Hence up to replacing $K$ with a finite extension and conjugating with an appropriate $m\in \text{PGL}_2(K)$ we can assume that $\phi=f/g$ with $f,g\in K[x]$ coprime, $\deg f>\deg g$ and $f$ monic. Let $S$ be a finite set of primes of $K$ such that $\alpha$ and all coefficients of $f$ and $g$ are $S$-integers. Notice that if $\beta\in \phi^{-\infty}(\alpha)$ and $\mathfrak P$ is a prime of $K(\beta)$ not lying above any prime of $S$, then $v_{\mathfrak P}(\beta)\geq 0$, where $v_{\mathfrak P}$ is the corresponding valuation. This is easy to see inductively: it holds for any $\beta\in \phi^{-1}(\alpha)$ because $\beta$ is a solution of $f(x)-\alpha g(x)=0$, that is a monic polynomial of degree $d$ with $S$-integral coefficients, and the inductive step follows for the analogous reason. 

Now enlarge $S$ by adding all the finitely many primes of $K$ that lie above rational primes that ramify in $K$.

Let $\p$ be a prime of $K$ not lying in $S$, suppose that $\p$ ramifies in $K^{\ab}_\infty(\phi,\alpha)$ and let $p$ be the rational prime lying below $\p$. We will show that $p$ is bounded in terms of $\phi$, $\alpha$ and $K$. Clearly there exists an abelian point $\beta\in \phi^{-\infty}(\alpha)$ such that $\p$ ramifies in $K(\beta)$ but does not ramify in $K(\phi(\beta))$. Let $\Delta\in K(\phi(\beta))$ be the discriminant of the polynomial $f(x)-\phi(\beta)g(x)$. Notice that since $\p\notin S$ then $\Delta$ is an $S'$-integer, where $S'$ is the set of primes of $K(\phi(\beta))$ lying above $S$.

The assumptions on the ramification, together with the key fact that $K(\beta)/K$ is a Galois extension, show that if $\mathfrak P$ is any prime of $K(\phi(\beta))$ lying above $\p$, then $|\Delta|_{\mathfrak P}<1$. Indeed, this follows from the fact that any prime  of $K(\phi(\beta))$ that ramifies in $K(\beta)$ divides the relative discriminant of the extension $K(\phi(\beta))/K(\beta)$, and since the extension  is Galois, any prime ideal in $K(\phi(\beta))$ that lies over $\p$ ramifies in $K(\beta)$.

If we normalize $|\cdot |_{\mathfrak P}$ so that $|p|_{\mathfrak P}=1/p$, then we have

$$|\Delta|_{\mathfrak P}\leq p^{-1}$$
and therefore
$$
\prod_{\mathfrak P\mid \p}|\Delta|_{\mathfrak P}^{[K(\phi(\beta))_{\mathfrak P}:K_\p]} \leq p^{-[K(\phi(\beta)):K]},
$$
so that
\begin{equation}\label{height1}
h(\Delta)=h(\Delta^{-1})\geq \frac{1}{[K(\phi(\beta)):\Q]}\cdot [K(\phi(\beta)):K]\cdot \log p=\frac{\log p}{[K:\Q]}.
\end{equation}
Now notice that $\Delta$ is a polynomial expression in $\phi(\beta)$, whose coefficients and degree only depend on $\phi$. It follows immediately that
\begin{equation}\label{height2}
h(\Delta)\leq  C_1\cdot h(\phi(\beta))+C_2, 
\end{equation}
where $C_1$ and $C_2$ are constants that depend only on $\phi$, see~\cite[Theorem~3.11]{silverman2}. On the other hand, recall that $\phi(\beta)\in \phi^{-\infty}(\alpha)$, and the height of elements in $\phi^{-\infty}(\alpha)$ is bounded in terms of $\phi$ and $h(\alpha)$; this follows from the fact that $h(\phi(z))\geq \deg\phi\cdot h(z)+C_3$ for every $z\in \overline{K}$, where $C_3$ is a constant that depends only on $\phi$. Hence,~\eqref{height1} and~\eqref{height2} imply that $p$ is upper bounded in terms of $K$, $\phi$ and $h(\alpha)$, which concludes the proof.
%
\end{proof}
\begin{remark}
One might be tempted to reprove Theorem \ref{finite_ramification}, using exactly the same argument, for what could be called ``Galois part'' of $K_\infty(\phi,\alpha)$, namely the extension $K'$ of $K$ generated by all points in $\phi^{-\infty}(\alpha)$ that generate a Galois extension of $K$. However, the argument we used does not work because if $\p$ is a prime of $K$ that ramifies in $K'$, it is not necessarily true that one can find a point $\beta\in \phi^{-\infty}(\alpha)$ with the properties that $K(\beta)/K$ is Galois, $\p$ ramifies in $K(\beta)$ and $\p$ does not ramify in $K(\phi(\beta))$, since $K(\phi(\beta))/K$ is not necessarily Galois (while if $K(\beta)/K$ is abelian then $K(\phi(\beta))/K$ is abelian as well). On the other hand, the argument can be applied verbatim to show that if \emph{every} $\beta\in \phi^{-\infty}(\alpha)$ is Galois over $K$ then $K_\infty(\phi,\alpha)$ is ramified only at finitely many primes.
\end{remark}

Recall that if $\phi\in K(x)$ and $\alpha\in K$ we say that the pair $(\phi,\alpha)$ is \emph{eventually stable} if, writing $\phi^n=p_n/q_n$ for coprime $p_n,q_n\in K[x]$, the number of irreducible factors of $p_n-\alpha q_n$ is eventually constant as $n\to+\infty$. We say that $(\phi,\alpha)$ is \emph{stable} if $p_n-\alpha q_n$ is irreducible for every $n$.

\begin{remark}\label{ev_stable_eq}
A pair $(\phi,\alpha)$ is eventually stable if and only if there exists some $n\geq 1$ such that for every $\beta\in \phi^{-n}(\alpha)$ the pair $(\phi,\beta)$ is stable over $K(\beta)$. This is an easy application of Capelli's Lemma (for a proof see~\cite[Lemma~0.1]{FeiSch}), which states that if $f,g$ are polynomials over a field $K$ then $f\circ g$ is irreducible over $K$ if and only if $f$ is irreducible over $K$ and $g-\beta$ is irreducible over $K(\beta)$ for every root $\beta$ of $f$. The analogous statement holds when $f\in K[x]$ and $\phi=p/q\in K(x)$, where $p,q$ are coprime polynomials, that is, the numerator of $f\circ \phi$ is irreducible if and only if $f$ is irreducible over $K$ and $p-\beta q$ is irreducible over $K(\beta)$ for every root $\beta$ of $f$. Therefore, if, for some $n\ge 1$, the pair $(\phi,\beta)$ is stable for every $\beta\in \phi^{-n}(\alpha)$, then, if $f$ is the minimal polynomial of $\beta$ over $K$, $f$ is \emph{$\phi$-stable}, namely the numerator of $f\circ \phi^m$ is irreducible for every $m$. On the other hand $\phi^{-n}(\alpha)$ is finite, showing that $p_n-\alpha q_n$ factors as a product of $\phi$-stable factors, and in turn implying that $(\phi,\alpha)$ is eventually stable. Conversely, if $(\phi,\alpha)$ is eventually stable then there exists $n\geq 1$ such that $p_n-\alpha q_n=f_1\ldots f_r$ for some $r\geq 1$ and for every $m\geq n$ the polynomial $p_m-\alpha q_m$ factors as a product of $r$ irreducible factors as well. Now one can check that the $f_i$'s are $\phi$-stable, and this implies that $(\phi,\beta)$ is stable over $K(\beta)$ for every root $\beta$ of some $f_i$.
\end{remark}

It is conjectured in~\cite{jones_ev} that eventual stability is the generic condition, as follows.

\begin{conjecture}{{\cite[Conjecture~1.2]{jones_ev}}}\label{all_es}
Let $K$ be a number field, let $\phi\in K(x)$ and let $\alpha\in K$. The pair $(\phi,\alpha)$ is eventually stable if and only if $\alpha$ is not periodic for $\phi$.
\end{conjecture}

We now have the following consequence of Theorem~\ref{finite_ramification}.

\begin{corollary}\label{abelian_implies_pcf}
Let $K$ be a number field, let $\phi\in K(x)$ and let $\alpha\in K$ be non-exceptional.
\begin{enumerate}
\item If  $G_\infty(\phi,\alpha)$ is abelian, then $\phi$ is PCF.
\item If  $[K_\infty^{\ab}(\phi,\alpha):K]=\infty$ and $(\phi,\alpha)$ is eventually stable, then $\phi$ is PCF.
\item If Conjecture~\ref{all_es} holds true and $[K_\infty^{\ab}(\phi,\alpha):K]=\infty$, then $\phi$ is PCF.
\end{enumerate}
\end{corollary}
\begin{proof}
(1) We have $K_\infty(\phi,\alpha)=K_\infty^{\ab}(\phi,\alpha)$, and hence $K_\infty(\phi,\alpha)/K$ is only ramified at finitely many primes by Theorem~\ref{finite_ramification}. It follows by~\cite[Theorem~5]{bridy} that $\phi$ is PCF.

(2) By Remark~\ref{ev_stable_eq}, eventual stability of the pair $(\phi,\alpha)$ implies the existence of some $n\geq 1$ such that if $\beta\in \phi^{-n}(\alpha)$ then the pair $(\phi,\beta)$ is stable over $K(\beta)$. Now since $\bigcup_{\beta\in \phi^{-n}(\alpha)} \phi^{-\infty}(\beta)$ differs from $\phi^{-\infty}(\alpha)$ by a finite set, the fact that $[K_\infty^{\ab}(\phi,\alpha):K]=\infty$ and the pigeonhole principle ensure the existence of some $\beta$ such that $\phi^n(\beta)=\alpha$ and such that $[K(\beta)_\infty^{\ab}(\phi,\beta)\colon K(\beta)]$ is infinite. Since $(\phi,\beta)$ is stable over $K(\beta)$, this implies that $K(\beta)_\infty(\phi,\beta)=K(\beta)_\infty^{\ab}(\phi,\beta)$ because for every $m$ the elements in $\phi^{-m}(\beta)$ live in the same $G_{K(\beta)}$-orbit, and hence they are all abelian. Now part (1) applies and we can conclude.

(3) If $\alpha$ is not periodic for $\phi$, then by Conjecture~\ref{all_es} the claim follows from part (2). Hence we can assume that $\alpha$ is periodic for $\phi$. Let $n$ be the exact period of $\alpha$ and let $\{\alpha=\alpha_0,\alpha_1,\ldots,\alpha_{n-1}\}$ be the forward orbit of $\alpha$. For every $i=0,\ldots,n-1$, we let $S_i=\{\gamma\in \phi^{-1}(\alpha_i)\colon \gamma \mbox{ is not periodic for }\phi\}$. Notice that at least one $S_i$ is non-empty, since otherwise $\alpha$ would be exceptional. Now let $F$ be the compositum $\prod_{i=1}^n\prod_{\gamma\in S_i}K(\gamma)_\infty^{\ab}(\phi,\gamma)$. We claim that $K_\infty^{\ab}(\phi,\alpha)\subseteq F$. In fact if $\beta\in \phi^{-\infty}(\alpha)$ is an abelian point that does not lie in $K$ then $\beta$ cannot be periodic for $\phi$, since it eventually enters a periodic orbit that is entirely contained in $K$. Hence there is some maximal $m\geq 0$ such that $\phi^m(\beta)$ is not periodic for $\phi$. For the same reason we must have $\phi^{m+1}(\beta)=\alpha_i$ for some $i\in\{0,\ldots,n-1\}$, and therefore $\gamma_i\coloneqq \phi^m(\beta)\in S_i$. This implies, in particular, that $\beta\in K(\gamma_i)_\infty^{\ab}(\phi,\gamma_i)\subseteq F$, since $\beta$ is certainly abelian over $K(\gamma_i)$, too.

Since by hypothesis $[K_\infty^{\ab}(\phi,\alpha):K]=\infty$, this shows that there exist $j\in \{0,\ldots,n-1\}$ and $\gamma_j\in S_j$ such that $K(\gamma_j)_\infty^{\ab}(\phi,\gamma_j)$ is an infinite abelian extension of $K$, and hence also an infinite abelian extension of $K(\gamma_j)$ in particular. But then we can use Conjecture~\ref{all_es} and part (2) again and conclude, since by construction $\gamma_j$ is not periodic for $\phi$.
\end{proof}


\section{Cyclotomic points for quadratic maps}\label{cyclo_quadratic}

When studying abelian points in backward orbits, it is a natural starting point to focus on the simplest ones, namely cyclotomic points.

\begin{definition}
Let $K$ be a number field, $\phi\in K(x)$ and $\alpha\in K$. We say that $\beta\in \phi^{-\infty}(\alpha)$ is a \emph{cyclotomic point} if it belongs to $K(\zeta_N)$ for some $N$-th root of unity $\zeta_N$. We call \emph{cyclotomic part} of $K_\infty(\phi,\alpha)$, and we denote it by $K_\infty^c(\phi,\alpha)$, the extension of $K$ generated by all cyclotomic points of $\phi^{-\infty}(\alpha)$.
\end{definition}

Remark~\ref{abelian_part_properties} holds verbatim for the cyclotomic part.

Cyclotomic points have been studied in~\cite{ostafe}, following ideas of Dvonicich and Zannier~\cite{DZ}, for a rather general class of rational maps. The results of~\cite{ostafe} suggest that having infinitely many cyclotomic points in a backward orbit is a very restrictive condition, and only functions that are conjugate to power or Chebyshev maps should provide instances of this phenomenon. We recall that a degree $d$ rational function $\phi\in K(x)$ is called \emph{special} if $\phi$ is $K^c$-conjugate to $x^{\pm d}$ or to $\pm T_d(x)$, where $T_d(x)$ is the Chebyshev polynomial of degree $d$.

\begin{definition}
Let $K$ be a number field, $\phi\in K(x)$ and $\alpha\in K$. We say that $(\phi,\alpha)$ is a \emph{special pair} if there exists a root of unity $\zeta$ such that $(\phi,\alpha)$ is $K^c$-conjugate to $(x^{\pm d},\zeta)$ or $(\pm T_d(x),\zeta+\zeta^{-1})$.
\end{definition}

We also take this opportunity to correct the definition of special rational function in~\cite[Definition~1.1]{ostafe} (and consequently also in~\cite[Definition~1.3]{OSSZ}), with the one above where the conjugacy of rational functions is with respect to $K^c$ rather than $K$.

We remark that being special is a notion that depends on the ground field as well, and not only on the function. For example, the function $2x^4$ is not special when seen as an element of $\Q(x)$, but it is special when seen in $\Q(\sqrt[3]{2})(x)$. To better explain why we ask for conjugacy over $K^c$, rather than over $\overline{K}$, we prove the following lemma, that is a straightforward extension of~\cite[Theorems~12 and~13]{andrews}.

\begin{lemma}\label{conjugacy_field}
Let $K$ be a number field, let $d\ge 2$ and let $\psi\in \{x^{\pm d},\pm T_d(x)\}$.
\begin{enumerate}
\item Let $\beta\in K$ be non-exceptional for $\psi$. Then the following are equivalent.
\begin{enumerate}[(a)]
\item $[K^c_\infty(\psi,\beta):K]=\infty$;
\item $[K^{\ab}_\infty(\psi,\beta):K]=\infty$;
\item There exists a root of unity $\zeta$ such that the pair $(\psi,\beta)$ is either $(x^{\pm d},\zeta)$ or $(\pm T_d(x),\zeta+\zeta^{-1})$;
\item $K_\infty(\psi,\beta)=K_\infty^{\ab}(\psi,\beta)=K_\infty^c(\psi,\beta)$.
\end{enumerate}
\item Let $\phi$ have degree $d$ and let $\alpha\in K$. Suppose that $\phi$ is $\overline{K}$-conjugate to $\psi$. Then $[K^c_\infty(\phi,\alpha):K]=\infty$ if and only if $(\phi,\alpha)$ is a special pair.
\end{enumerate}
\end{lemma}
\begin{proof}
(1) The implications $(a) \implies (b)$, $(c)\implies (d)$ and $(d)\implies (a)$ are obvious, using the fact that $T_d(x+x^{-1})=x^d+x^{-d}$. The implication $(b)\implies (c)$ is essentially proved in~\cite{andrews} for $x^d$ and $T_d(x)$. Although the hypothesis in~\cite{andrews} is that $K_\infty(\psi,\beta)=K_\infty^{\ab}(\psi,\beta)$, rather than just $[K^{\ab}_\infty(\psi,\beta):K]=\infty$, the only property that is needed for the implication is that there exists a sequence $\{\beta_n\}_{n\geq 0}$ entirely contained in $K^{\ab}$ and such that $\beta_0=\beta$ and $\psi(\beta_n)=\beta_{n-1}$ for every $n\geq 1$. It is easy to construct iteratively such a sequence with the sole hypothesis that $[K^{\ab}_\infty(\psi,\beta):K]=\infty$: since $\psi^{-1}(\beta)$ is finite, there must be some $\beta_1\in \psi^{-1}(\beta)$ such that $\beta_1$ is an abelian point and $[K(\beta_1)_\infty^{\ab}:K(\beta_1)]=\infty$. Now the same reasoning applies to $\psi^{-1}(\beta_1)$, and so on.

To prove the implication for $\psi=x^{-d}$, just consider $\psi^2=x^{d^2}$. To prove it for $-T_d(x)$, notice that when $d$ is even, then $(-T_d)^n=-(T_d)^n(x)$, and thus this case follows from~\cite{andrews}. When $d$ is odd, then $(-T_d)^n(x)=(-1)^nT_d^n(x)$. However, since  $[K^{\ab}_\infty(\psi,\beta):K]=\infty$ implies $[K^{\ab}_\infty(\psi^2,\beta):K]=\infty$ and since $\psi^2=T_{d^2}$, the conclusion follows again from~\cite{andrews} applied with  $\psi^2$.
%

(2) The ``if'' part is obvious. Conversely, let $f=\frac{ax+b}{cx+d}$ be a function in $\PGL_2(\overline{K})$ such that $\psi=f\circ\phi\circ f^{-1}$, and suppose that $[K^c_\infty(\phi,\alpha):K]=\infty$. Let $\beta\coloneqq f(\alpha)$. Notice that $f^{-1}$ yields a bijection between $\psi^{-\infty}(\beta)$ and $\phi^{-\infty}(\alpha)$. If $f$ is defined over a finite extension $L/K$, it follows that $\psi^{-\infty}(\beta)$ contains infinitely many points in $L^c$.

Now assume that $\psi=x^{\pm d}$. By (1), $\beta$ is a root of unity, and therefore $\psi^{-\infty}(\beta)$ only contains roots of unity. Let $\zeta_1,\zeta_2,\zeta_3$ be three distinct elements in $\psi^{-\infty}(\beta)$. Since $f^{-1}$ is a bijection between $\psi^{-\infty}(\beta)$ and $\phi^{-\infty}(\alpha)$, we have $f^{-1}(\zeta_i)\in K^c$ for every $i=1,2,3$. Since $f^{-1}$ is determined uniquely by the images of three points, we deduce that $f^{-1}$ is defined over $K^c$.

If $\psi=\pm T_d(x)$, by (1) there is a root of unity $\zeta$ such that $\beta=\zeta+\zeta^{-1}$. Since $T_d(x+x^{-1})=x^d+x^{-d}$ for every $x$, it follows that all points in $\psi^{-\infty}(\beta)$ are of the form $\mu+\mu^{-1}$ where $\mu$ is a root of unity. Indeed, it is enough to look at $\psi^{-1}(\beta)$, since the conclusion will follow by induction. Let $\gamma\in \psi^{-1}(\beta)$, that is,  $\psi(\gamma)=\gamma^d+\gamma^{-d}=\zeta+\zeta^{-1}$. Seeing this as a quadratic equation in $\gamma^d$, we conclude that $\gamma^d$ has to be $\zeta$ or $\zeta^{-1}$. Hence $\psi^{-\infty}(\beta)\subseteq K^c$, and the same argument of the $x^{\pm d}$ case allows to conclude.
\end{proof}

As mentioned in the introduction, Conjecture~\ref{conjecture} states that if $K$ is a number field, $\phi\in K[x]$ is a polynomial of degree $d\ge 2$ and $\alpha\in K$ is not an exceptional point for $\phi$, then $G_\infty(\phi,\alpha)$ is abelian if and only if the pair $(\phi,\alpha)$ is $K^{\ab}$-conjugate to $(x^d,\zeta)$ or $(\pm T_d(x),\zeta+\zeta^{-1})$ for a root of unity $\zeta$. When $K=\Q$ the abelian part of $K_\infty(\phi,\alpha)$ coincides with its cyclotomic one, and therefore the problem studied in~\cite{ostafe} is closely related to Conjecture~\ref{conjecture}.

The following lemma follows in a straightforward way from the results of~\cite{ostafe}.

\begin{lemma}\label{superattracting_point}
Let $\phi\in K(x)$ have degree $d\ge 2$, and suppose that $\phi$ has a periodic critical point. Then $[K^c_\infty(\phi,\alpha):K]=\infty$ if and only if $(\phi,\alpha)$ is a special pair.
\end{lemma}
\begin{proof}
We first note by Remark~\ref{abelian_part_properties} that $[K^c_\infty(\phi,\alpha):K]=\infty$ if and only if there are infinitely many cyclotomic points in $\phi^{-\infty}(\alpha)$. If $(\phi,\alpha)$ is a special pair, then by definition it follows immediately that there are infinitely many cyclotomic points in $\phi^{-\infty}(\alpha)$, and thus $[K^c_\infty(\phi,\alpha):K]=\infty$.

Therefore, we discuss the converse implication. We assume that there are infinitely many cyclotomic points in $\phi^{-\infty}(\alpha)$ and we want to conclude that $(\phi,\alpha)$ is a special pair.

By the hypothesis that $\phi$ has a periodic critical point, there exists $n\geq 1$ such that $\phi^n$ has a superattracting fixed point. Up to conjugating $\phi^n$, possibly over a finite extension $L$ of $K$, we can assume that such superattracting fixed point is the point at infinity. But that means precisely that $\phi^n$ can be written as $f/g$ with $f,g\in L[x]$ and $\deg f>\deg g+1$. 

We can now  proceed exactly as in the proof of~\cite[Theorem~1.2]{ostafe} to conclude that $\phi^n$ is $\overline{K}$-conjugate to $x^{dn}$ or $\pm T_{dn}(x)$. Indeed, by our assumption we obtain that there are infinitely many $\beta\in K^c$ such that $\phi^{n\ell}(\beta)=\alpha$ for some $\ell\ge 1$. This is equivalent to having infinitely many $\beta\in K^c$ such that $\phi_\alpha^{n\ell}(\beta)=1$ for some $\ell\ge 1$, where $\phi_\alpha=\alpha^{-1}\phi(\alpha x)$. We note that $\phi_\alpha^n$ satisfies same degree property as $\phi^n$, and thus applying exactly the same proof of~\cite[Theorem~1.2]{ostafe}\footnote{We take the occasion to remark that the formulation of~\cite[Corollary]{FZ}, that is used 
in~\cite[Theorem~1.2]{ostafe}, is incorrect, as $\pm x^d$ should be replaced by $x^{\pm d}$.} (with $h$ replaced by $\phi_\alpha^n$, looking only at the preimages of the root of unity $1$ rather than any root of unity, see also Remark~\ref{rem:DZ HIT} below), we conclude that $\phi_\alpha^n$ is $\overline{K}$-conjugate to $x^{dn}$ or $\pm T_{dn}(x)$. We also refer to \cite{ferra5} for an exposition of the ideas used in the proof of \cite[Theorem~1.2]{ostafe}.

Now it is a classic result in arithmetic dynamics (see for example~\cite{silverman4} or~\cite[Theorem~1.7]{silverman2}) that if $\phi^n$ is $\overline{K}$-conjugate to a polynomial for some $n$, then $\phi$ is itself $\overline{K}$-conjugate to a polynomial or to $1/x^d$. If $\phi$ is $\overline{K}$-conjugate to a polynomial we can re-apply the aforementioned  proof of~\cite[Theorem~1.2]{ostafe} to $\phi$ and get that $\phi$ is in fact $\overline{K}$-conjugate to $x^d$ or $\pm T_d(x)$. Now we apply Lemma~\ref{conjugacy_field} to conclude.
\end{proof}

\begin{remark}\label{rem:DZ HIT}
In the proof of Lemma~\ref{superattracting_point}, we have used the proof of~\cite[Theorem~1.2]{ostafe}. We note that the extra condition
 in~\cite[Theorem~1.2]{ostafe} on the rational function $\phi$ is to apply a cyclotomic version of the Hilbert Irreducibility Theorem~\cite[Corollary~2]{DZ}, which is not needed in our present case.
\end{remark}

\begin{remark} 
We remark that the same argument that proves Lemma~\ref{superattracting_point} can be used to obtain a generalization of~\cite[Theorem~2]{DZ} to rational functions with a periodic critical point. 
It is a question of independent interest to get an explicit value of the constant $c$ in the result of  Loxton~\cite[Theorem~1]{loxton}, which is one of the main tools in the proof of~\cite[Theorem~2]{DZ}. This, together with an effective version of the torsion coset theorem due to Aliev and Smyth~\cite{aliev}, would lead to an explicit bound on the number of cyclotomic preperiodic points in the result of Dvornicich and Zannier~\cite[Theorem~L]{DZ}, as well as explicit versions
of the results of~\cite{Chen,ostafe}.
\end{remark}

Lemma~\ref{superattracting_point} immediately implies the following corollary.
\begin{corollary}\label{ap_conj_polys}
Let $\phi\in \Q[x]$ be a polynomial of degree $d\ge 2$ and let $\alpha\in \Q$ be non-exceptional for $\phi$. Then Conjecture~\ref{conjecture} holds true for $(\phi,\alpha)$.
\end{corollary}
\begin{proof}
Just notice that $\Q^{\ab}=\Q^c$, and that $\infty$ is a fixed critical point for any polynomial of degree at least two.
\end{proof}

Of course it is a natural question how to extend Conjecture~\ref{conjecture} to all rational functions. As we will see in Section~\ref{lattes}, one cannot hope to prove in full generality that $G_\infty(\phi,\alpha)$ is abelian if and only if the pair $(\phi,\alpha)$ is special, as Latt\`es maps provide non-special examples of rational functions with infinitely many abelian points (but not infinitely many cyclotomic points) in backward orbits. However, Corollary~\ref{abelian_implies_pcf} implies that in order for $G_\infty(\phi,\alpha)$ to be abelian it is necessary for $\phi$ to be PCF, and being PCF is a very restrictive condition. In particular, the main result of~\cite{benedetto2} implies that for a given number field $K$ and a given $d\ge 2$, there exists a finite set $S$ of PCF rational maps in $K(x)$ of degree $d$ such that any PCF map in $K(x)$ of degree $d$ is either $\overline{K}$-conjugate to a flexible Latt\`es map or to an element of $S$.

For quadratic rational functions in $\Q(x)$, the set $S$ has been determined explicitly in~\cite{lukas} (it consists of 12 conjugacy classes) and of course there are no $\Q$-Latt\`es maps of degree 2 with rational coefficients. This allows to prove the following result for rational functions in the spirit of Conjecture~\ref{conjecture}.

\begin{theoremA}\label{ap_conjecture}
Let $\phi\in \Q(x)$ be a non-special rational map of degree $d\geq 2$, let $\alpha\in \Q$ be a non-exceptional point for $\phi$ and suppose that $G_\infty(\phi,\alpha)$ is abelian. Then the following hold.
\begin{enumerate}
\item $\phi$ is PCF and all of its critical points are strictly preperiodic.
\item If $d=2$, then $\phi$ is $\Q$-conjugate to either $\frac{2x^2+4x+4}{-x^2+4}$ or $\frac{1}{2(x-1)^2}$.
\end{enumerate}\end{theoremA}

\begin{proof}
(1) By Corollary~\ref{abelian_implies_pcf}, $\phi$ must be PCF. By Lemma~\ref{superattracting_point}, no critical point can be periodic since $\phi$ is not special.

In order to prove (2), first use again the fact that, by Corollary~\ref{abelian_implies_pcf}, the map $\phi$ is PCF. Now the authors of~\cite{lukas} classified all 12 $\overline{\Q}$-conjugacy classes of quadratic PCF rational functions, by giving one explicit representative for each class. Moreover, they show that the only ones with non-trivial twists (i.e.\ the only ones that admit $\overline{\Q}$-conjugates that are not $\Q$-conjugates) are $x^2$ and $1/x^2$, and these are treated in Lemma~\ref{conjugacy_field}. We refer to~\cite[Table~1]{lukas} for a list of representatives with their post-critical portraits. Hence it is enough for our purposes to study the remaining 10 representatives. Two of the remaining classes are represented by polynomials: $x^2-2$ and $x^2-1$. The first one is the Chebyshev polynomial of degree 2 (which is special and thus excluded by our hypothesis), and the second one fulfils the hypothesis of Lemma~\ref{superattracting_point} (and moreover it has been proven in~\cite{ferra1} by completely different arguments that $G_\infty(x^2-1,\alpha)$ is nonabelian for every $\alpha\in \Q$).

Of the remaining 8 classes, 4 of them have at least one periodic critical point. Hence, they fulfil again the hypotheses of Lemma~\ref{superattracting_point}. The last four classes represent PCF rational functions whose critical points are both strictly preperiodic. Representatives for these classes are: $\frac{2}{(x-1)^2}$, $\frac{2x^2+4x+4}{-x^2+2}$, $\frac{2x^2+4x+4}{-x^2+4}$ and $\frac{1}{2(x-1)^2}$. To conclude the proof, we have to show that $G_\infty(\phi,\alpha)$ is not abelian for any $\alpha\in \Q$ when $\phi=\frac{2}{(x-1)^2}$ or $\frac{2x^2+4x+4}{-x^2+2}$. This has been done via computations in Magma~\cite{magma}, as we now explain. Let $t$ be an indeterminate over $\Q$. First, let $\phi=\frac{2}{(x-1)^2}$, and let $M$ be the splitting field of $\phi^5-t$ over $\Q(t)$. Using the Magma intrinsic {\tt GaloisSubgroup} one can verify that the constant field of $M$ contains $\Q(i,\sqrt{2+2\sqrt{2}})$, that is a normal non-abelian extension of $\Q$. Now it is a standard fact in the theory of function fields that if $\alpha\in \Q$ and  $P_{\alpha}$ is the associated place of degree $1$ of $\Q(t)$, then the decomposition group of any place of $M$ lying above $P_{\alpha}$ surjects onto the Galois group of the constant field of $M$ (see for example \cite[Proposition 9.8]{rosen}). On the other hand, when $P_\alpha$ does not ramify in $M$ then the decomposition group of any place of $M$ lying above $P_\alpha$ is isomorphic to the specialized Galois group $\gal(\phi^5-\alpha)$ (for a proof, see for example~\cite[Lemma~2]{neftin}). Hence, for every $\alpha\in \Q$ such that $P_\alpha$ is unramified in $M$ we have that $G_5(\phi,\alpha)$ is non-abelian, since it surjects onto $\gal(\Q(i,\sqrt{2+\sqrt{2}})/\Q)$, and therefore in particular $G_\infty(\phi,\alpha)$ is non-abelian. Now it is just a simple finite computation to verify that $\gal(\phi^5-\alpha)$ is non-abelian for the ramified $\alpha$'s, that are just $\alpha=0,2$.\footnote{We remark that if $F$ is the constant field of $M$, it is also true that for all but finitely many $\alpha\in \Q$ the splitting field of $\phi^5-\alpha$ over $\Q$ contains $F$, as can be argued by writing a generator for $F$ in terms of the roots of $\phi^5-t$ and then specializing. However, finding the exceptional $\alpha$'s in this case would require a tough computation, since the splitting field of $\phi^5-t$ has large degree.}

The case of $\phi=\frac{2x^2+4x+4}{-x^2+2}$ is treated in the same way: one checks, using the same intrinsic, that the splitting field of $\phi^4-t$ contains again $\Q(i,\sqrt{2+2\sqrt{2}})$, so that $G_4(\phi,\alpha)$ cannot be abelian for unramified $\alpha$'s. The only ramified value in this case is $\alpha=-2$.
\end{proof}

It is plausible that even when $\phi=\frac{2x^2+4x+4}{-x^2+4}$ or $\frac{1}{2(x-1)^2}$ the group $G_\infty(\phi,\alpha)$ cannot be abelian for any $\alpha\in \Q$; however in order to use the same trick we used in the proof of Theorem~\ref{ap_conjecture} it is not sufficient to consider the first five iterates of $\phi$, because for those the constant field of the splitting field of $\phi^n-t$ is abelian. We remark that extending Lemma~\ref{superattracting_point} to all rational functions would allow to prove directly that if $\phi\in \Q(x)$ and $\alpha\in \Q$ is not exceptional for $\phi$, then $G_\infty(\phi,\alpha)$ is abelian if and only if the pair $(\phi,\alpha)$ is special.

\vspace{3mm}
Pairs $(\phi,\alpha)$ with $\phi\in\Q(x)$ and $\alpha\in \Q$ that fulfil the hypotheses of Lemma~\ref{superattracting_point} enjoy a stronger property than the one assumed by Theorem~\ref{ap_conjecture}: it is enough for them to have $[\Q_\infty^c(\phi,\alpha):\Q]=\infty$, rather than $\Q_\infty^c(\phi,\alpha)=\Q_\infty(\phi,\alpha)$, in order to be special. This is not entirely surprising, as according to Conjecture~\ref{all_es} most of them are expected to be eventually stable. As shown in the proof of Corollary~\ref{abelian_implies_pcf}, if $(\phi,\alpha)$ is eventually stable and $[\Q_\infty^c(\phi,\alpha):\Q]=\infty$, then there exists some $\beta\in \phi^{-\infty}(\alpha)$ such that $\Q(\beta)_\infty(\phi,\beta)=\Q(\beta)_\infty^{\ab}(\phi,\beta)$.

To conclude this section, we will show that for certain classes of degree two rational functions $\phi$ over a number field $K$ that do not necessarily fulfil the hypotheses of Lemma~\ref{superattracting_point}, the extension $K_\infty^c(\phi,\alpha)/K$ is still finite for any $\alpha$.

Notice that if $\phi\in K(x)$ is a rational function of degree 2 not conjugate to a polynomial, then $\phi$ is conjugate over $\overline{K}$ to a function of the form $ax+\frac{a}{x}+c$ for some $a\neq 0$.

First, recall the following result from~\cite{jones_ev} and the notation therein. Let $v$ be a discrete valuation on a field $K$ and denote by $\p$ the associated prime ideal $\{t \in K: v(t)>0\}$ of the ring $R=\{t\in K: v(t)\ge 0\}$. Let $k$ be the residue field $R/\p$. We denote by $\overline{t}\in\PP^1(k)$ the reduction modulo $\p$ of $t\in \PP^1(K)$, and we denote by $\overline{\phi}$ the rational function obtained from $\phi\in K(x)$ by reducing each coefficient modulo $\p$. We say that $\phi$ has \emph{good reduction} at $\p$ if $\deg \phi=\deg\overline{\phi}$.

\begin{theorem}[{{\cite[Theorem~4.6]{jones_ev}}}]\label{ev_stability_sufficient}
Let $K$ be a field and let $v$ be a discrete valuation on $K$ with valuation ring $R$, maximal ideal $\p$ and residue field $k$. Let $\phi\in K(x)$ have degree $d\ge 2$ and let $\alpha\in K$. Suppose that $\phi$ has good reduction at $\p$, $\phi(\alpha)\neq \alpha$ and $\overline{\phi}^{-1}(\overline{\alpha})=\{\overline{\alpha}\}$, where $\overline{\phi}$ is seen as a self-rational map of $\mathbb P^1_k$. Then $(\phi,\alpha)$ is eventually stable.
\end{theorem}

We also need the following result.

\begin{lemma}
\label{lem:roots of 1}
Let $K$ be a number field, let $\phi\in K(x)$ and let $\alpha\in K$. If $p>d=\deg\phi$ is a prime such that $p^{N-1}$ does not divide $[K:\Q]$ for some $N\geq 2$, and $\zeta_{p^N}$ is a primitive $p^N$-th root of unity, then $\zeta_{p^N}\notin K_\infty(\phi,\alpha)$.
\end{lemma}

\begin{proof}
We  first notice that for every $n$ the group $G_n(\phi,\alpha)$ is a subgroup of $S_d\wr\ldots\wr S_d$ ($n$-fold wreath product), and $|S_d\wr\ldots\wr S_d|$ is a power of $d!$ (see for example~\cite{ferra2}). Hence if $\ell$ is a prime dividing $|G_n(\phi,\alpha)|$ then $\ell\leq d$. 

We now assume $p>d$, and hence $p\nmid |G_n(\phi,\alpha)|$ for any $n$. On the one hand, if $p^{N-1}$ does not divide $[K:\Q]$ for some $N\geq 2$ (and hence $\zeta_{p^N}\notin K$), then $p$ divides $ |\gal(K(\zeta_{p^N})/K)|=[K(\zeta_{p^N}):K]$. On the other hand, if $\zeta_{p^N}\in K_\infty(\phi,\alpha)$, then by Galois theory there is a surjection $G_\infty(\phi,\alpha)\to \gal(K(\zeta_{p^N})/K)$. This surjection must factor through some $G_n(\phi,\alpha)$, but this is impossible because $p\nmid |G_n(\phi,\alpha)|$.
\end{proof}

We are now ready to prove our last main result of this section.

\begin{theoremA}\label{suff_conditions}
Let $K$ be a number field, let $a,c\in K$ with $a\neq 0$ and let $\phi\coloneqq ax+a/x+c$. Let $\alpha\in K$, and suppose that there exists a prime $\p$ of $K$ lying above $2$ such that one of the following conditions hold:
\begin{enumerate}
\item $v_\p(a)=0$ and $v_\p(c)>0$;
\item $v_\p(a)=0$, $h(a)>\log 4$ and $c\equiv 1\bmod \p$.
\end{enumerate}
Then $[K^c_\infty(\phi,\alpha):K]<\infty$.
\end{theoremA}

\begin{proof}
Suppose by contradiction that $[K^c_\infty(\phi,\alpha):K]=\infty$. By Theorem~\ref{finite_ramification}, the extension $K^c_\infty(\phi,\alpha)/K$ is ramified at a finite set $S$ of primes. Moreover, by Lemma~\ref{lem:roots of 1} if $p$ is a rational odd prime lying below a prime of $S$ then there exists a non-negative integer $N_p$ such that $K^c_\infty(\phi,\alpha)\cap K(\zeta_{p^{\infty}})=K(\zeta_{p^{N_p}})$. Notice that since $S$ is finite, the set of $N_p$'s is finite as well. Hence up to adding a finite number of roots of unity to $K$ we can assume that:
$$K^c_\infty(\phi,\alpha)=K(\zeta_{2^{\infty}}).$$


By Abhyankar's Lemma there exists a large enough positive integer $\ell$ with the following properties:
\begin{itemize}
\item $\zeta_{2^\ell}\notin K$;
\item if $n\geq \ell$ and $\mathfrak q$ is a prime of $K(\zeta_{2^n})$ lying above $2$, then $\mathfrak q$ ramifies in $K(\zeta_{2^{n+1}})$.
\end{itemize}
From now on, if $\beta\in \phi^{-\infty}(\alpha)$ we will write $\beta_n$ to denote $\phi^n(\beta)$. Let $N\ge\ell+1$. Since $K^c_\infty(\phi,\alpha)=K(\zeta_{2^{\infty}})$, there must be some $\beta\in \phi^{-\infty}(\alpha)$ such that $K(\beta)=K(\zeta_{2^{N+1}})$ and $K(\beta_1)=K(\zeta_{2^N})$.

Let us show that if $(1)$ or $(2)$ hold and $\P$ is a prime of $K(\beta_1)$ lying above $\p$, then we must have $v_\P(\beta_1)=0$. Suppose by contradiction that $v_\P(\beta_1)\neq 0$. Our assumptions on $a,c$ clearly imply that $v_\P(\beta_n)<0$ for every $n\geq 2$, and hence $v_\P(\beta_n)\neq 0$ for every $n\geq 1$. Now let $m\geq 1$ be such that  $K(\beta_m)=K(\zeta_{2^N})$ and $K(\beta_{m+1})=K(\zeta_{2^{N-1}})$. Then the minimal polynomial of $\beta_m$ over  $K(\beta_{m+1})$ is $x^2+\frac{c-\beta_{m+1}}{a}x+1$, and the choice of $N$ ensures that the prime $\P'$ of $K(\beta_{m+1})$ lying below $\P$ ramifies in $K(\beta_m)$. Now on the one hand $\beta_m\cdot \sigma(\beta_m)=1$, where $\sigma$ generates $\gal(K(\beta_m)/K(\beta_{m+1}))\cong C_2$, and on the other hand $\beta_m$ and $\sigma(\beta_m)$ must have the same valuation at $\P$, since this is the only prime lying above $\P'$. This contradicts the fact that $v_\P(\beta_m)\neq 0$.

Let then $\P$ be a prime of $K(\beta_1)$ lying above $\p$ and assume that (1) or (2) hold, so that $v_\P(\beta_1)=0$. Since $K(\beta)=K(\zeta_{2^{N+1}})$ and $K(\beta_1)=K(\zeta_{2^N})$, it follows that the minimal polynomial of $\beta$ over $K(\beta_1)$ is $x^2+\frac{c-\beta_1}{a}x+1$, and its discriminant is $\Delta=\left(\frac{c-\beta_1}{a}\right)^2-4$. Since $\P$ ramifies in $K(\beta)$ it must be that $v_\P(\Delta)\neq 0$.

If (1) holds, we cannot have $v_\P(\Delta)\neq 0$, since $v_\P(a)=v_\P(\beta_1)=0$ while $v_\P(4),v_\P(c)>0$.

If (2) holds, in order to have $v_\P(\Delta)\neq 0$ we must have $c\equiv \beta_1\equiv 1\bmod \P$. Let us show that the pair $(\phi,\beta_1)$ satisfies the hypotheses of Theorem~\ref{ev_stability_sufficient} with respect to the valuation induced by $\P$. First, notice that the choice of $N$ implies that it cannot be $\phi(\beta_1)=\beta_1$, as otherwise we would have $\beta_n=\beta_1$ for every $n\geq 1$ and consequently $\zeta_{2^N}\in K$. Next, denote by $\overline{\cdot}$ the reduction modulo $\P$, and consider the reduced map $\overline{\phi}$. This has degree $2$, meaning that $\phi$ has good reduction at $\P$. Now look at the equation $\overline{\phi}(x)=\overline{\beta}_1$: this becomes $x^2+1=0$, which has only $1=\overline{\beta}_1$ as a root. Hence the conditions of Theorem~\ref{ev_stability_sufficient} are fulfilled and $(\phi,\beta_1)$ is eventually stable over $K(\beta_1)$. Corollary~\ref{abelian_implies_pcf} implies then that $\phi$ is PCF. To conclude, notice that $\infty$ is a fixed point for $\phi$. A quick computation shows that its multiplier is $1/a$. Now~\cite[Corollary 1.3]{benedetto2} yields a contradiction, as it shows that multipliers of fixed points of PCF quadratic maps have height bounded above by $\log 4$.
 \end{proof}

\section{Abelian points for Latt\`es maps}\label{lattes}

In this section, we will give necessary and sufficient conditions on a triple $(K,\phi,\alpha)$, where $K$ is a number field, $\phi\in K(x)$ is a $K$-Latt\`es map and $\alpha\in K$, for $G_\infty(\phi,\alpha)$ to be abelian (or, equivalently, for $\phi^{-\infty}(\alpha)$ to only consist of abelian points). Notice that Latt\`es maps are never $\overline{K}$-conjugate to a special map (see~\cite[Exercise~6.8]{silverman2}).

Throughout the section, we will use the following standard notation: for an elliptic curve $E$ over a field $K$ we denote by $O_E$ the neutral element for the group structure on $E$; if $m\geq 1$ is an integer we denote by $E[m]$ the subgroup of $m$-torsion points of $E$ and $E(L)_{tor}$ the subgroup of $L$-rational torsion points, for any extension $L/K$.

Recall that for a number field $K$, a $K$-Latt\`es map is a rational function $\phi\in K(x)$ such that there exist an elliptic curve $E/K$, a rational point $T\in E(K)$, an endomorphism $\vartheta\in \en_K(E)$ and a non-constant map $\pi\colon E\to \mathbb P^1$ such that the following diagram commutes:
\begin{equation}\label{lattes_diagram}
\xymatrix{
    E \ar[r]^{\psi} \ar[d]_{\pi} & E \ar[d]^{\pi} \\
    \mathbb P^1 \ar[r]^{\phi}       & \mathbb P^1 },
    \end{equation}
    where $\psi(Q)=\vartheta(Q)+T$ for every $Q\in E$. It is possible to prove (see~\cite[Proposition~6.77]{silverman2}) that $T$ must be a torsion point of $E$.
    
When $E$ is in Weierstrass form, we can assume without loss of generality that $\pi$ is one of the following maps:
$$\pi(x,y)=\begin{cases}x & j(E)\mbox{ arbitrary}\\
x^2 & j(E)=1728\\
y & j(E)=0\\
x^3 & j(E)=0\\
			\end{cases}.$$
			
Notice that changing the Weierstrass model of $E$ has the effect of conjugating the map $\phi$, and the conjugation takes place over $K$. This is a simple consequence of the fact that all maps involved are defined over $K$. For example, assume that $\pi=x$. A change of coodinates to a Weierstrass model $E'$ has the form $F\colon (x,y)\mapsto (u^2x+r,u^3y+u^2sx+t)$, with $r,s,t,u\in K$ and $u\ne 0$. Letting $f(x)=u^2x+r\in K(x)$ one gets the following diagram, whose squares all commute:
$$\xymatrix{
    E' \ar[r]^{F^{-1}} \ar[d]_{x} & E \ar[r]^{\psi} \ar[d]^{x} & E \ar[d]^{x} \ar[r]^{F} & E'\ar[d]^{x} \\
    \mathbb P^1 \ar[r]^{f^{-1}}& \mathbb P^1 \ar[r]^{\phi}       & \mathbb P^1\ar[r]^{f} & \mathbb P^1 }.$$

We can therefore assume that $E$ is given by an equation in short Weierstrass form. Finally, the preperiodic points of $\phi$ are exactly those of the form $\pi(P)$, where $P\in E(\overline{K})_{tor}$. For proofs of these assertions, see~\cite[Chapter~6]{silverman2}.

The main result of this section is the following theorem.

\begin{theoremA}\label{lattes_abelian}
Let $K$ be a number field, let $\phi\in K(x)$ be a $K$-Latt\`es map of degree $d\geq 2$ associated to an elliptic curve $E$ and let $\alpha\in K$. Then $\phi^{-\infty}(\alpha)\subseteq K^{\ab}$ if and only if the following three conditions hold:
\begin{enumerate}
\item $\alpha$ is preperiodic for $\phi$;
\item $E$ has CM by an order in a quadratic number field $F$;
\item the field $F=\text{End}(E)\otimes \Q$ is contained in $K$.
\end{enumerate}
\end{theoremA}

In order to prove Theorem \ref{lattes_abelian}, we first introduce a number of preliminary results. The first one is the following elementary lemma.

\begin{lemma}\label{isogenies_shape}
Let $E\colon y^2=x^3+Ax+B$ be an elliptic curve over a number field $K$. Let $\vartheta\colon E\to E$ be an isogeny defined over $K$. Then $\vartheta$ can be written in the following form:

$$\vartheta(x,y)=\begin{cases}(f_1(x),y\cdot f_2(x)) & j(E) \mbox{ arbitrary}\\
(x\cdot f_1(x^2),y\cdot f_2(x^2)) & j(E)=1728\\
(x\cdot f_1(x^3),y\cdot f_2(x^3) & j(E)=0\\
(x\cdot f_1(y),f_2(y)) & j(E)=0\\
			\end{cases},$$
for some $f_1,f_2$ rational functions in one variable with coefficients in $K$.
\end{lemma}
\begin{proof}
The proof simply uses the fact that $\vartheta$ commutes with automorphisms of $E$. To prove the first line, write $\vartheta(x,y)=(g_1(x)+yg_2(x),h_1(x)+yh_2(x))$ for some $g_1,g_2,h_1,h_2\in K(x)$. Since $\vartheta(-P)=-\vartheta(P)$ for every $P\in E(\overline{K})$, it follows immediately that $g_2(x)=h_1(x)=0$.

Now using the fact that we can write $\vartheta(x,y)=(f_1(x),y\cdot f_2(x))$, the second and the third lines follow by the fact that $\vartheta\circ \sigma=\sigma\circ\vartheta$, where $\sigma$ is the automorphism of $E$ that sends $(x,y)\mapsto (-x,iy)$ and $(x,y)\mapsto (\zeta_3x,y)$ in the second and third case, respectively (see for example~\cite[Proposition~6.26]{silverman2}). Here $i$ is a primitive $4$-th root of unity and $\zeta_3$ a primitive $3$-rd root of unity.

The last line follows from the third because when $j(E)=0$ then $x^3=y^2-B$.
\end{proof}

The next preliminary result follows from the results of~\cite{lozano}. Hence we will now recall the setting and the notation of the aforementioned paper.

Let $F$ be an imaginary quadratic number field with discriminant $\Delta_F$ and ring of integers $\O_F$, let $f\geq 1$ be an integer and $\O_{F,f}$ be the associated order, namely the subring $\Z f\tau+\Z$. Let $\tau\coloneqq \frac{\Delta_F+\sqrt{\Delta_F}}{2}$. The pair $\{f\tau,1\}$ is a $\Z$-basis of $\O_{F,f}$. We set $\varphi\coloneqq \Delta_Ff$ and $\delta\coloneqq \frac{\Delta_F(1-\Delta_F)}{4}f^2$. We denote by $j_{F,f}$ the $j$-invariant of the complex lattice $\C/\O_{F,f}$.

For every integer $N\geq 3$ we define the following Cartan subgroup of $\GL_2(\Z/N\Z)$:
$$\mathcal C_{\delta,\varphi}(N)\coloneqq \left\{\left(\begin{array}{cc}a+b\varphi & b\\ \delta b & a\end{array}\right): a,b\in \Z/N\Z, a^2+ab\varphi-\delta b^2\in (\Z/N\Z)^{\times}\right\}$$
and its normalizer
$$\mathcal N_{\delta,\varphi}(N)\coloneqq \left\langle \mathcal C_{\delta,\varphi}(N),\left(\begin{array}{cc}-1 & 0 \\ \varphi & 1\end{array}\right)\right\rangle.$$

Moreover, we let $\mathcal C_{\delta,\varphi}(p^\infty)$ and $\mathcal N_{\delta,\varphi}(p^{\infty})$ be the inverse limits of $\mathcal C_{\delta,\varphi}(p^n)$ and $\mathcal N_{\delta,\varphi}(p^n)$, respectively.

For a prime $p$ and an elliptic curve $E$ over a number field $K$, we denote by $T_p(E)$ the $p$-adic Tate module of $E$ and, once a $\Z_p$-basis of $T_p(E)$ is chosen, we denote by $\rho_{E,p}\colon G_K\to \GL_2(\Z_p)$ the associated $p$-adic Galois representation.

\begin{theorem}[{{\cite[Theorem~1.2]{lozano}}}]\label{cm_images}
With the notation above, let $E/\Q(j_{F,f})$ be an elliptic curve with CM by $\O_{F,f}$ and let $p$ be a prime number. Then there exists a basis of $T_p(E)$ such that the image of the Galois representation $\rho_{E,p}\colon G_{\Q(j_{F,f})}\to \GL_2(\Z_p)$ is contained in $\mathcal N_{\delta,\varphi}(p^\infty)$, and has finite index therein. Moreover, if $\sigma\in G_{\Q(j_{F,f})}$ generates $\gal(F(j_{F,f})/\Q(j_{F,f}))$, then $\rho_{\Q(j_{F,f})}(\sigma)\notin \mathcal C_{\delta,\varphi}(p^\infty)$.
\end{theorem}
One can use Theorem \ref{cm_images} to prove the following result, that we will need later on.
\begin{lemma}\label{commutators}
Let $K$ be a number field and let $E/K$ be an elliptic curve with CM by an order $\O_{F,f}$ in a quadratic number field $F$. Let $p$ be a prime number and let $\rho_{E,p}\colon G_K\to \GL_2(\Z_p)$ be the Galois representation attached to the $p$-adic Tate module of $E$. If $F$ is not contained in $K$, then the commutator subgroup $[\rho_{E,p}(G_K),\rho_{E,p}(G_K)]$ contains elements of infinite order.
\end{lemma}
\begin{proof}
First, let us show that it is enough to prove the following statement:
\begin{center}
(*) if $G\leq \mathcal N_{\delta,\varphi}(p^\infty)$ is a finite index subgroup not contained in $\mathcal C_{\delta,\varphi}(p^\infty)$, then the commutator subgroup $[G,G]$ contains elements of infinite order.
\end{center}
In fact, assume that (*) holds true, let $j_{F,f}$ be the $j$-invariant of $E$, and let $E'/\Q(j_{F,f})$ be an elliptic curve with CM by $\O_{F,f}$ and $j$-invariant $j_{F,f}$. By hypothesis $K$ contains $\Q(j_{F,f})$, and of course $E$ and $E'/K$ are twists of each other via a $1$-cocycle $\chi\colon G_K\to \aut(E)$. 
Let $L\coloneqq \overline{K}^{\ker\chi}$; the curves $E/L$ and $E'/L$ are isomorphic, and the restrictions of their Galois representations to $G_L$ are isomorphic as well. By Theorem~\ref{cm_images}, up to changing the basis of $T_p(E)$ we can assume that $\rho_{E,p}$ and $\rho_{E',p}$ coincide as maps on $G_L$ and $H\coloneqq \rho_{E',p}(G_L)=\rho_{E,p}(G_L)$ is a finite index subgroup of $\mathcal N_{\delta,\varphi}$. Hence if $H$ is not contained in $\mathcal C_{\delta,\varphi}$ we can apply claim (*) with $G=H$ and conclude. If, on the other hand, $H\subseteq \mathcal C_{\delta,\varphi}$ then since $[\mathcal N_{\delta,\varphi}(p^\infty):H]<\infty$, the group $H$ must contain a non-diagonal element. Since $G_L$ is a normal subgroup of $G_K$, $\rho_{E,p}(G_K)$ is contained in the normalizer of $\rho_{E,p}(G_L)$ inside $\GL_2(\Z_p)$, and~\cite[Lemma~5.1]{lozano} shows that we must have $\rho_{E,p}(G_K)\subseteq \mathcal N_{\delta,\varphi}(p^{\infty})$. Notice that we cannot have $\rho_{E,p}(G_K)\subseteq \mathcal C_{\delta,\varphi}(p^\infty)$, as otherwise $\rho_{E,p}(G_K)$ would be abelian and hence $F\subseteq K$ by Theorem~\ref{cm_images}, contradicting the hypothesis. Hence in this case we can apply claim (*) with $G=\rho_{E,p}(G_K)$.

Hence we reduced to proving (*). From now on, we let $G$ be a finite index subgroup of $\mathcal N_{\delta,\varphi}$ not containing $\mathcal C_{\delta,\varphi}$. Notice that for every $n\geq 1$ there is an isomorphism of groups
$$M_n\colon (\O_{F,f}/p^n\O_{F,f})^{\times}\to \mathcal C_{\delta,\varphi}(p^n)$$
where $M_n(z)$ is the multiplication-by-$z$ matrix with respect to the $\Z$-basis $\{f\tau,1\}$ of $\O_{F,f}/p^n\O_{F,f}$ (see~\cite[Lemma~2.5]{lozano}). This of course yields an isomorphism $$M\colon\varprojlim_n(\O_{F,f}/p^n)^{\times}\to \mathcal C_{\delta,\varphi}(p^\infty),$$
and we will identify groups that are isomorphic via $M_n$ or $M$. Clearly there is an injection
$$\iota\colon U_p\coloneqq\{\alpha\in \O_{F,f}\colon p\nmid N_{F/\Q}(\alpha)\}\hookrightarrow\mathcal C_{\delta,\varphi}(p^\infty)$$
whose image is dense. Now let $\sigma$ be the generator of $\gal(F/\Q)$. We make the following claim:

\begin{center}
(**) there exist $\alpha\in U_p$ such that $\iota(\alpha)\in G$ and $\iota(\alpha)\iota(\sigma(\alpha))^{-1}$ has infinite order.
\end{center}

To prove it, notice that if $\alpha\in U_p$ is such that $\iota(\alpha)\iota(\sigma(\alpha))^{-1}$ has finite order $m\geq 1$, then $\alpha^m\equiv \sigma(\alpha)^m\bmod p^n$ for every $n\geq 1$; therefore it must be $\alpha=\zeta_m\sigma(\alpha)$ for some $m$-th root of unity $\zeta_m$, and in turn this implies $m=1,2,3,4,6$ since $F$ is imaginary quadratic. Now for a fixed $m\in \{1,2,3,4,6\}$ let $U_{p,m}$ denote the set of all $\alpha\in U_p$ such that $\alpha=\xi_m\sigma(\alpha)$ for some $m$-th root of unity $\xi_m$. Let $\pi_n$ denote the natural surjection $U_p\twoheadrightarrow (\O_{F,f}/p^n\O_{F,f})^{\times}$. One can easily verify that $|\pi_n(U_{p,m})|=o(|(\O_{F,f}/p^n\O_{F,f})^{\times}|)$\footnote{Here $o(\cdot)$ is the usual little-$o$ notation, namely for functions $f,g\colon\N\to \R$ we write $f=o(g)$ if $f/g\to 0$ when $n\to \infty$.} and hence
\begin{equation}\label{smallo}
\sum_{m\in\{1,2,3,4,6\}}|\pi_n(U_{p,m})|=o(|(\O_{F,f}/p^n\O_{F,f})^{\times}|).
\end{equation}
This boils down to computing an upper bound on $|\pi_n(U_{p,m})|$ and then comparing it to the value $|(\O_{F,f}/p^n\O_{F,f})^{\times}|$ given in~\cite[Lemma~2.5]{lozano}. Let us show how one proceeds in the $m=4$ case; the other ones are completely analogous. If $m=4$ it must be $F=\Q(i)$. Let $a+bf\tau\in U_{p,m}$. The conditions $\alpha=\zeta_m\sigma(\alpha)$ for some $\zeta_m\in \{\pm1,\pm i\}$ are equivalent to the four conditions
$$b=0,\, \,\,\, a=bf, \, \,\, a=3bf, \, \, \, a=4bf.$$
The set of matrices $\left(\begin{array}{cc}a+b\varphi & b\\ \delta b & a\end{array}\right)\in \mathcal C_{\delta,\varphi}(p^n)$ that respect one of these four conditions clearly has at most $4p^{n-1}(p-1)$ elements, and this value is $o(|(\O_{F,f}/p^n\O_{F,f})^{\times}|)$ by~\cite[Lemma~2.5]{lozano}. Repeating this argument for the remaining $m$'s yields a proof of~\eqref{smallo}. On the other hand, since $[\mathcal N_{\delta,\varphi}(p^\infty):G]<\infty$, then $[\mathcal C_{\delta,\varphi}(p^\infty): G\cap \mathcal C_{\delta,\varphi}(p^\infty)]<\infty$ and this means that the projection $G_n$ of $G\cap \mathcal C_{\delta,\varphi}(p^\infty)$ on $(\O_{F,f}/p^n\O_{F,f})^{\times}$ has eventually constant index in $(\O_{F,f}/p^n\O_{F,f})^{\times}$. Since for every $n$ the map $\pi_n$ is surjective, by~\eqref{smallo} for $n$ large enough there must be elements of $(\O_{F,f}/p^n\O_{F,f})^{\times}$ that lie in $G_n\setminus\pi_n\left(\bigcup_{m\in\{1,2,3,4,6\}} U_{p,m}\right)$. This shows that there are elements of $G\cap \mathcal C_{\delta,\varphi}(p^\infty)$ that do not belong to $\iota\left(\bigcup_{m\in\{1,2,3,4,6\}}U_{p,m}\right)$. Notice though that $\iota(U_{p,m})$ is closed for every $m$, and hence the intersection of $G\cap \mathcal C_{\delta,\varphi}(p^\infty)$ (that is a finite index closed subgroup of $\mathcal C_{\delta,\varphi}(p^\infty)$, and hence is open therein) and the complement of $\iota\left(\bigcup_{m\in\{1,2,3,4,6\}}U_{p,m}\right)$ is a nonempty open subset of $\mathcal C_{\delta,\varphi}(p^\infty)$, and hence by density it intersects $\iota(U_p)$, proving (**).

Now all that remains to do is to prove claim (*). Let $\varepsilon\coloneqq \left(\begin{array}{cc}-1 & 0\\ \varphi & 1\end{array}\right)$, and notice that since $[\mathcal N_{\delta,\varphi}(p^\infty):\mathcal C_{\delta,\varphi}(p^\infty)]=2$ and by hypothesis $G\not\subseteq \mathcal C_{\delta,\varphi}(p^\infty)$ , there is an element $g\in G$ that lies in the same coset of $\varepsilon$ with respect to $\mathcal C_{\delta,\varphi}(p^\infty)$, and hence there is some $\beta\in \mathcal C_{\delta,\varphi}(p^\infty)$ such that $\beta\varepsilon \in G$. Now let $\alpha\in U_p$ be an element as in (**). Then 
$$[\iota(\alpha),\beta\varepsilon]=\iota(\alpha)\cdot\beta\cdot(\varepsilon \cdot \iota(\alpha)\cdot \varepsilon)^{-1}\cdot \beta^{-1}=\iota(\alpha)\iota(\sigma(\alpha))^{-1},$$
by the commutativity of $\mathcal C_{\delta,\varphi}(p^\infty)$ and the fact that $\varepsilon\cdot \iota(\alpha)\cdot \varepsilon=\iota(\sigma(\alpha))$. By (**) it follows that the commutator $[\iota(\alpha),\beta\varepsilon]$ has infinite order, and the proof of (*) is complete.
\end{proof}


Next, we state and prove a technical lemma that allows us to reduce to the study of $K$-Latt\`es maps with $T=O_E$, i.e.\ $K$-Latt\`es maps coming from endomorphisms of elliptic curves. Let then $K$ be a number field, $\phi\in K(x)$ be a $K$-Latt\`es map of degree $d\geq 2$, and let $E/K$ be an elliptic curve, $\vartheta\in \en_K(E)$, $T\in E(K)_{tor}$ and $\pi\colon E\to \mathbb P^1$ be the data that define diagram~\eqref{lattes_diagram}, so that $\psi(Q)=\vartheta(Q)+T$ for every $Q\in E$. Of course $\vartheta$ also has degree $d$.

For any morphism $\mu\colon E\to E$ of degree $\ge 2$ and any $Q\in E$, we will denote by $T_\infty(\mu,Q)$ the infinite regular rooted tree of degree $\deg\mu$ defined in the obvious way: the root of the tree is $Q$, the points at distance $n$ from the root are the elements of $\mu^{-n}(Q)=\{P\in E\colon \mu^n(P)=Q\}$. The \emph{level} of a point $P\in T_\infty(\mu,Q)$ is the distance of $P$ from the root $Q$. A node $P$ at level $n+1$ is connected to a node $P'$ at level $n$ if $\mu(P)=P'$ (notice that such tree is always regular since $\mu$ is unramified).

Our map $\psi$ as above defines, toghether with a point $P\in E$, a tree $T_\infty(\psi,P)$ of degree $d=\deg\psi=\deg \vartheta\geq 2$ which surjects, via $\pi$, to $\phi^{-\infty}(\pi(P))$.

Now let $m\geq 1$ be an integer such that $mT=O_E$. We let $m\cdot T_\infty(\psi,P)$ be the tree constructed in the following way: first we multiply by $m$ every node of $T_\infty(\psi,P)$,  then at every level we identify equal nodes and finally we identify edges with the same ends, so that between two nodes in $m\cdot T_\infty(\psi,P)$ there is always at most one edge. Clearly there is a level preserving surjection $[m]\colon T_\infty(\psi,P)\twoheadrightarrow m\cdot T_\infty(\psi,P)$. Moreover, if $Q,Q'\in m\cdot T_\infty(\psi,P)$ are nodes at level $n,n-1$, respectively, then one sees easily that they are connected by an edge if and only if $\vartheta(Q)=Q'$. However, notice that the tree $m\cdot T_\infty(\psi,P)$ in general is only a proper subtree of $T_\infty(\vartheta,mP)$, since multiplication by $m$ might have the effect of collapsing points on the same level, whenever $(m,d)\neq 1$. 

\begin{lemma}\label{endomorphism_reduction}
Let $\phi,E,\vartheta,\pi,m$ be as above and let $P\in E(K)$. Then the following hold.
\begin{enumerate}[(a)]
\item There exists an integer $N\geq 1$ such that if $P'\in T_\infty(\psi,P)$ has level $N$, then every node of the tree $T_\infty(\vartheta^N,mP')$ belongs to $m\cdot T_\infty(\psi,P)$. Moreover, $P$ is torsion if and only if $mP'$ is torsion.
\item If $\pi(T_\infty(\psi,P))\subseteq K^{\ab}$ then $\pi(m\cdot T_\infty(\psi,P))\subseteq K^{\ab}$ and the extension $K(m\cdot T_\infty(\psi,P))/(K(m\cdot T_\infty(\psi,P))\cap K^{\ab})$ has finite exponent.
\end{enumerate}
 
\end{lemma}

\begin{proof}
(a) Let $N\geq 1$ be an integer such that $\ker \vartheta^n\cap E[m]=\ker\vartheta^N\cap E[m]$ for every $n\geq N$, let $\ker \vartheta^N\cap E[m]=\{R_1,\ldots,R_k\}$ and let $Q\in T_\infty(\psi,P)$ be a point at some level $n\geq N$. We claim that all the elements of $(\vartheta^N)^{-1}(mQ)$ belong to $m\cdot T_\infty(\psi,P)$. To see this, first notice that since $Q$ is at distance $n\geq N$ from the root, then for every $i\in \{1,\ldots,k\}$ also $Q+R_i$ is a node of $T_\infty(\psi,P)$ at distance $n$ from the root. Clearly $m(Q+R_i)=mQ$ for every $i\in \{1,\ldots,k\}$, which implies that all the $Q+R_i$'s collapse to one single point in $m\cdot T_\infty(\psi,mP)$. Now let $Q',S_1,\ldots,S_k\in E$ be such that $\psi^N(Q')=Q$ and $\vartheta^N(S_i)=R_i$ for every $i$. Then for every $i$ we have that $\psi^{-N}(Q+R_i)=\{Q'+S_i+U\colon U\in \ker\vartheta^N\}$, and clearly $[m](\psi^{-N}(Q+R_i))$ has $d^N/k$ distinct elements. Hence if we prove that for every $i\neq j$ we have $[m](\psi^{-n}(Q+R_i))\cap [m](\psi^{-n}(Q+R_j))=\varnothing$, we will have proved that $mQ$ has $k\cdot d^N/k=d^N$ distinct preimages via $\vartheta^N$ in $m\cdot T_\infty(\psi,mP)$, i.e.\ our claim. So suppose by contradiction that there exist $U_1,U_2\in \ker\vartheta^N$ such that $[m](Q'+S_i+U_1)=[m](Q'+S_j+U_2)$ for some $i\neq j$. Then $[m](S_i-S_j+U_1-U_2)=O_E$, i.e.\ $S_i-S_j+U_1-U_2\in E[m]$. On the other hand, since $\vartheta^N(S_i)=R_i$ and $\vartheta^N(S_j)=R_j$, and $R_i,R_j\in \ker\vartheta^n$, it follows that $S_i,S_j\in \ker\vartheta^{N+n}$. But $\ker\vartheta^N\cap E[m]=\ker\vartheta^{N+n}\cap E[m]$ and hence we must have $S_i=S_j+R_h$ for some $h\in \{1,\ldots,k\}$. It follows immediately that $\{Q'+S_i+U\colon U\in \ker\vartheta^N\}=\{Q'+S_j+U\colon U\in \ker\vartheta^N\}$, which is impossible since points in the LHS are ancestors via $\psi^N$ of $Q+R_i$ and points in the RHS are ancestors of $Q+R_j$. This argument shows that if we choose any point $P'\in T_\infty(\psi,P)$ at level $N$, then all the nodes of the tree $T_\infty(\vartheta^N,mP')$ belong to $m\cdot T_\infty(\psi,P)$.

The assertion on $P$ and $mP'$ is obvious, since $\psi^N(P')=P$ and $T$ is torsion.

(b) Let $Q\in m\cdot T_\infty(\psi,P)$ and write $Q=mR$ for some $R\in T_\infty(\psi,P)$; by hypothesis $\pi(R)\in K^{\ab}$. Let $R=(x_0,y_0)$ and write the multiplication-by-$m$ map on $E$ in the shape given by Lemma~\ref{isogenies_shape} according to $\pi$: the first claim becomes immediately clear. For example, if $\pi=x^2$ write the multiplication-by-$m$ map as $(x\cdot f_1(x^2),y\cdot f_2(x^2))$ for some $f_1,f_2\in K(x)$; then $x_0^2\in K^{\ab}$ by hypothesis and therefore $\pi(Q)=\pi(x_0\cdot f_1(x_0^2),y_0\cdot f_2(x_0^2))=x_0^2f_1(x_0^2)^2\in K^{\ab}$, too.

To prove the second claim, notice that it is enough to show that there exists a positive integer $e$ such that for every $Q\in m\cdot T_\infty(\psi,P)$ we have $[K(Q):K(Q)\cap K^{\ab}]\le e$. In fact, $K(m\cdot T_\infty(\psi,P))$ is the compositum of all $K(Q)$'s. But this is obvious, in fact if $Q=(x_Q,y_Q)$ lives on $E\colon y^2=x^3+Ax+B$, then $y_Q^2=x_Q^3+Ax_Q+B$, and since $\pi(Q)\in K^{\ab}$ it is clear that $[K(Q):K(Q)\cap K^{\ab}]\leq 12$ in any case.
\end{proof}

\begin{corollary}\label{diagram_reduction}
Let $\phi,E,\vartheta,\pi,m$ be as above, and let $\alpha\in K$. If $G_\infty(\phi,\alpha)$ is abelian, then there exists an integer $N\geq 1$ and an element $\alpha'\in K^{\ab}$ such that if $\phi'$ is the $K$-Latt\`es map that fits in the following diagram:
$$
\xymatrix{
    E \ar[r]^{\vartheta^N} \ar[d]_{\pi} & E \ar[d]^{\pi} \\
    \mathbb P^1 \ar[r]^{\phi'}       & \mathbb P^1 }
$$
then $G_\infty(\phi',\alpha')$ is a subgroup of $\gal(K^{\ab}/K(\alpha'))$ (and hence in particular it is abelian), and $\alpha$ is preperiodic if and only if $\alpha'$ is preperiodic.
\end{corollary}
\begin{proof}
Let $N$ be as in point $(a)$ of Lemma~\ref{endomorphism_reduction}, let $P'\in T_\infty(\psi,P)$ have distance $N$ from the root $P$ and let $\alpha'\coloneqq \pi(P')$. This belongs to $K^{\ab}$ by point $(b)$ and is preperiodic by $(a)$. By the same point $(a)$, $(\vartheta^N)^{-\infty}(\alpha')$ is entirely contained in $m\cdot \pi(T_\infty(\psi,P))$, and hence it entirely consists of abelian points by $(b)$.
\end{proof}

\begin{proof}[Proof of Theorem \ref{lattes_abelian}]
First, assume that $(1)$,$(2)$ and $(3)$ hold. Then $\alpha=\pi(P)$ for some $P\in E_{tor}$ and $F=\text{End}(E)\otimes \Q$ is a quadratic number field. Let $E\colon y^2=x^3+Ax+B$ for some $A,B\in K$. First assume that $AB\ne 0$. Then, as explained for example in \cite[Example 5.5.1]{silverman1}, the function $\pi=x \colon E\to \PP^1$ is a Weber function for the curve $E$. It follows by the theory of complex multiplication (see for example \cite[Corollary 5.7]{silverman1}) that $F(j(E),\pi(E_{tor}))$ is the maximal abelian extension of $F$. However $F\subseteq K$ by point $(3)$, and therefore by a standard fact in Galois theory $K(j(E),\pi(E_{tor}))=K(\pi(E_{tor}))$ is an abelian extension of $K$. But clearly $K(\phi^{-\infty}(\alpha))\subseteq K(\pi(E_{tor}))$, and the claim follows.

Now assume that $AB=0$. In these cases the field $F$ has class number $1$, since it is either $\Q(i)$ or $\Q(\zeta_3)$. Then, as proved in \cite[Example 5.8]{silverman1}, the field $F(E_{tor})$ is the maximal abelian extension of $F$, and therefore $K(E_{tor})/K$ is abelian by (3). Again, $K(\phi^{-\infty}(\alpha))\subseteq K(\pi(E_{tor}))$ and the claim follows.

Conversely, assume that $G_\infty(\phi,\alpha)$ is abelian. Let $\phi$ be defined via diagram~\eqref{lattes_diagram}. Notice that since points $(1)$ and $(2)$ in the statement are independent of the base field, in order to prove them we can assume (possibly enlarging $K$) that $\psi$ is an endomorphism of $E$, thanks to Corollary~\ref{diagram_reduction}. Let $P\in E$ be such that $\alpha=\pi(P)$. Now consider the tree $T_\infty(\psi,P)$; the map $\pi$ induces a surjection $T_\infty(\psi,P)\twoheadrightarrow \phi^{-\infty}(\alpha)$. In order to prove $(1)$ and $(2)$ we will prove the following claim:

\begin{center}
(*) there exists a finite extension $L/K$ and a point $P'\in T_\infty(\psi,P)$ defined over $L$ such that every point in $T_\infty(\psi,P')$ belongs to $E(L^{\ab})$.
\end{center}

Let us show first that claim (*) implies $(1)$ and $(2)$. Suppose it holds, and pick a sequence $\{P_n\}_{n\geq 0}$ of points in $T_\infty(\psi,P')$ such that $P_0=P'$ and $\psi^n(P_n)=P_0$ for every $n\geq 1$. To prove $(1)$, let $\widehat{h}_{E}$ be the N\'eron-Tate canonical height on $E$: we have that $\widehat{h}_E(P_n)\to 0$ as $n\to\infty$ because $\psi(P_n)=P_{n-1}$ and $\psi$ has degree $d\ge 2$. It follows by~\cite[Theorem~1]{silverman3} that $\widehat{h}_E(P_n)=0$ for every $n$, and therefore $P_0=P'$ is torsion, showing that $P$ is torsion, too (as $P=\psi^k(P')$ for some $k\geq 0$) and in turn that $\alpha$ is preperiodic for $\phi$. In order to prove $(2)$, suppose by contradiction that $E$ has no CM and let $p$ be a prime that divides $d$. Since $E$ has no CM the map $\psi$ is multiplication by an integer and hence $L(T_\infty(\psi,P'))$ contains $L(E[p^\infty])$. This implies that the $p$-adic Galois representation $\rho_{E,p}\colon G_L\to \GL_2(\Z_p)$ has abelian image, contradicting Serre's open image theorem according to which $\rho_{E,p}(G_L)$ is open in $\GL_2(\Z_p)$ (it is straigthforward to verify that $\GL_2(\Z_p)$ has no open abelian subgroups).

Let us now prove (*). This is done case by case accordingly to what map $\pi$ is, but the strategy is always the same. Start by assuming that $\pi$ is the $x$-coordinate map on $E\colon y^2=x^3+Ax+B$ where $A,B\in K$. Clearly there exists a point $P'\in T_\infty(\psi,P)$ such that $P'\notin E[2]$. Notice that therefore no point in $T_\infty(\psi,P')$ is $2$-torsion, since $\psi$ is an endomorphism. Let $L$ be the field of definition of $P'$. Now it is easy to show by induction that every point in $T_\infty(\psi,P')$ belongs to $L^{\ab}$: suppose it is true for every point at level $n$, and let $(x_{n+1},y_{n+1})$ have level $n+1$. Write $\psi$ in the form $(f_1(x),y f_2(x))$ for some $f_1,f_2\in K(x)$, via Lemma~\ref{isogenies_shape}. Then $\psi((x_{n+1},y_{n+1}))=(f_1(x_{n+1}),y_{n+1} f_2(x_{n+1}))$. Now since by hypothesis $G_\infty(\phi,\alpha)$ is abelian then $x_{n+1}=\pi((x_{n+1},y_{n+1}))\in L^{\ab}$ and hence $f_1(x_{n+1}),f_2(x_{n+1})\in L^{\ab}$. Since $\psi((x_{n+1},y_{n+1}))$ has distance $n$ from $P'$, by the inductive hypothesis $y_{n+1} f_2(x_{n+1})\in L^{\ab}$, and we are done because $f_2(x_{n+1})\neq 0$ since no point in the tree $T_\infty(\psi,P')$ is $2$-torsion.

Next, assume that $\pi=x^2$ and $E\colon y^2=x^3+Ax$ for some $A\in K$. Let $P'\in T_\infty(\psi,P)$ be such that no point in $T_\infty(\psi,P')$ is $2$-torsion, and let $L$ be the field of definition of $P'$. Again, assume by induction that every point at level $n$ in $T_\infty(\psi,P')$ is in $L^{\ab}$ and pick $(x_{n+1},y_{n+1})$ at level $n+1$. Write $\psi=(x f_1(x^2),y f_2(x^2))$ via Lemma~\ref{isogenies_shape}: then $\psi((x_{n+1},y_{n+1}))=(x_{n+1} f_1(x_{n+1}^2),y_{n+1} f_2(x_{n+1}^2))$ is at level $n$ and therefore $x_{n+1} f_1(x_{n+1}^2),y_{n+1} f_2(x_{n+1}^2)\in L^{\ab}$. By hypothesis $x_{n+1}^2=\pi((x_{n+1},y_{n+1}))\in L^{\ab}$ and on the other hand $f_1(x_{n+1}^2),f_2(x_{n+1}^2)\neq 0$ because $T_\infty(\psi,P')$ contains no $2$-torsion. It follows by the inductive hypothesis that $x_{n+1},y_{n+1}\in L^{\ab}$.

When $E\colon y^2=x^3+A$ for some $A\in K$ and $\pi=x^3$ or $y$, we pick $P'\in T_\infty(\psi,P)$ such that $T_\infty(\psi,P')$ contains no $2$-torsion nor $3$-torsion. Now we proceed in the same way we did for the other cases via Lemma~\ref{isogenies_shape}. The absence of $2$ or $3$-torsion makes the inductive step possible.

To conclude the proof we need to prove $(3)$. We know that $P$ is a torsion point, $E$ has CM by a quadratic field $F$ and $\psi$ is the composition of an endomorphism $\vartheta$ and the translation by a torsion point $T$. If $\vartheta$ is a CM endomorphism then it is defined over $KF$, and since it is defined over $K$ by hypothesis, it follows that $F\subseteq K$. Hence we can assume that $\vartheta$ is multiplication by the integer $e=\sqrt{d}$. Now let $m$ be an integer such that $mT=O$, choose an integer $N$ as in point $(a)$ of Lemma~\ref{endomorphism_reduction} and let $P'$ be at distance $N$ from $P$: we have that $T_\infty([e^N],mP')$ is entirely contained in $m\cdot T_\infty(\psi,P)$. This implies that if $p$ is a prime that divides $e$ then $K(E[p^{\infty}])\subseteq K(T_\infty([e^N],mP'))\subseteq K(m\cdot T_\infty(\psi,P))$. By point $(b)$ of Lemma~\ref{endomorphism_reduction}, the group $\gal(K(m\cdot T_\infty(\psi,P))/K)$ has a normal subgroup of finite exponent with abelian quotient. Hence the same must be true for $\gal(K(E[p^{\infty}]/K)$, that is a quotient of the former. However, this implies that all commutators in $\gal(K(E[p^{\infty}]/K)$ have finite order, contradicting Lemma~\ref{commutators} unless $F\subseteq K$.
\end{proof}

\begin{example}\label{lattes_example}
The simplest examples of maps satisfying the conditions of Theorem~\ref{lattes_abelian} are quadratic $K$-Latt\`es maps coming from endomorphisms (i.e.\ with $T=O_E$) with $K$ an imaginary quadratic number field. In these cases, the elliptic curve from which the map descends must be defined over $\Q$. Moreover, such maps come from elements of norm 2 in an order of $K$. There are only two instances of such behaviour: $\Q(i)$ and $\Q(\sqrt{-7})$, with the order being the full ring of integers. When $K=\Q(i)$, there are three families of quadratic $K$-Latt\`es maps coming from endomorphisms: $\frac{i}{2}\cdot\frac{x^2+A}{x}$, $-\frac{i}{2}\cdot\frac{x^2+A}{x}$ and $-\frac{1}{4}\cdot\frac{(x+A)^2}{x}$, where $A\in K^{\times}$. All maps in one single family are $K^{\ab}$-conjugate to each other. This follows from the theoretical reason that curves of the form $y^2=x^3+Ux$, for $U\in K^\times$, are all quartic twists of each other, but it can also easily be verified by hand. For example, if $A,A'\in K$ then the maps $\frac{i}{2}\cdot\frac{x^2+A}{x}$ and $\frac{i}{2}\cdot\frac{x^2+A'}{x}$ are conjugate via the linear transformation $x\mapsto \sqrt{A/A'}x$, that is defined over $K^{\ab}$. However it is false, in general, that two maps in the same family are $K$-conjugate to each other. Moreover, it can be easily checked by looking at multipliers that maps in one family are not $\overline{K}$-conjugate to maps in a different one. We remark that certain arithmetic dynamical properties of the map $\frac{i}{2}\cdot\frac{x^2-1}{x}$ have been studied in~\cite{mizusawa}.
\end{example}

\bibliographystyle{plain}
\bibliography{bibliography}

\end{document}